\documentclass[12pt,a4paper]{article}
\usepackage[dvipdfmx]{hyperref}
\usepackage{amsmath,amsthm,amssymb,enumerate,mathrsfs}
\usepackage[numbers]{natbib}

\def\E{\mathcal{E}}
\def\F{\mathcal{F}}

\def\M{\mathcal{M}}
\def\N{\mathbb{N}}

\def\R{\mathbb{R}}

\def\U{\mathcal{U}}
\def\K{{\mathcal{K}}}

\newcommand{\fn}{\footnote}

\theoremstyle{plain}
\newtheorem{thm}{Theorem}[section]
\newtheorem{cor}{Corollary}[section]
\newtheorem{lem}{Lemma}[section]
\newtheorem{prop}{Proposition}[section]
\newtheorem*{Fatou}{Fatou's lemma}
\newtheorem*{aux}{Auxiliary Theorem}

\theoremstyle{definition}
\newtheorem{assmp}{Assumption}[section]
\newtheorem{dfn}{Definition}[section]

\newtheorem{rem}{Remark}[section]

\numberwithin{equation}{section}
\allowdisplaybreaks

\title{Fatou's Lemma, Galerkin Approximations and the Existence of Walrasian Equilibria in Infinite Dimensions\thanks{This essay was written during Sagara's visit to Johns Hopkins University, August 13--21, 2016. The authors are grateful to Rich McLean for a  passing remark that shaped the introduction to the essay, and to  Takashi Suzuki  for wide-ranging conversations regarding Walrasian general equilibrium theory.   This research is supported by JSPS KAKENHI Grant No.\,26380246 from the Ministry of Education, Culture, Sports, Science and Technology, Japan.}}
\date{\today}
\author{M. Ali Khan\thanks{Corresponding author.}  \\{\small Department of Economics, The Johns Hopkins University} \\{\small Baltimore, MD 21218, United States} \\{\small e-mail: akhan@jhu.edu}\and \\
Nobusumi Sagara 
\\{\small Department of Economics, Hosei University} \\
{\small 4342, Aihara, Machida, Tokyo 194--0298, Japan} \\
{\small e-mail: nsagara@hosei.ac.jp}}

\begin{document}
\maketitle
\setcounter{page}{0}
\thispagestyle{empty}
\clearpage

\begin{abstract}
\noindent
This essay has three objectives: (i) to report recent generalizations of Fatou's lemma to multi-functions taking values in a Banach space, and framed in terms of both Bochner and Gelfand integration; (ii) to delineate the importance of Galerkin approximations in  Walrasian general equilibrium theory with a continuum of agents and commodities; and thereby (iii) to  present two new results on the existence of a Walrasian equilibrium in  economies where the continuum of agents is formalized as a saturated measure space. 
\end{abstract}


\tableofcontents
\clearpage

\begin{quote}
\textit{The drawing started to tremble because it wanted to go places.  All of a sudden most of art starts to tremble.  Cezanne was always trembling but very precisely.}\footnote{See \textit{The Collected Writings of Willem de Kooning}, Hanuman Press, Madras, pp.\,11, 26.}\hfill W. de Kooning (1949--1950)
\end{quote}

\begin{quote}
\textit{Although this may seem a paradox, all exact science is dominated by the idea of approximation. When a man tells you that he knows the exact truth about anything, you are safe in inferring that he is an inexact man.}\footnote{See  \textit{The Scientific Outlook}, George Allen \& Unwin Ltd.,  London, p.\,165.}\hfill B. Russell (1931)
\end{quote}

\section{Introduction}
Pierre Joseph Louis Fatou (1878--1929) was a French mathematician with an important lemma to his name, a lemma that is one of the cornerstones of  the elementary theory of Lebesgue integration. For a sequence of non-negative integrable functions whose integrals have a finite upper limit, one formulation of it  guarantees  an integrable  function  that is a pointwise upper limit of the sequence and whose integral is equal to or less than that of the upper limit of the integral of the sequence.  In terms of a symbolic expression and the use of notation that is by now conventional,  we have:  

\begin{Fatou}
Let $(T,\Sigma,\mu)$ be a measure space. If $\{ f_n \}_{n\in \N}$ is a sequence of non-negative integrable functions for which $\liminf_n\int f_nd\mu<\infty$, then 
$$
\int\liminf_{n\to\infty}f_nd\mu\le \liminf_{n\to \infty}\int f_nd\mu.
$$ 
\end{Fatou}

\noindent  To be sure, there are variations even in the standard  textbooks: the sequence of functions need not be defined everywhere but almost everywhere; it may not be constituted by non-negative functions; and even if so constituted, the range of the functions may be the extended real line.\fn{Among texts, Fremlin's is the most insistent on the requirement that the sequence of functions {\it not} be defined everywhere, and that doing so leads to pronounced ad-hockery in the more mature parts of the subject. The statement of the lemma presented above can be seen to have been taken from either the Fremlin or the Halmos texts, the difference lying in the meaning given to the terms \lq\lq integrable function."}    There are variations even in the proofs: the lemma is deduced as a straightforward implication of Lebesgue's monotone convergence theorem, or derived {\it ab initio} and used as an instrument for the proof of such a theorem.\fn{Royden's text on {\it Real Analysis} is perhaps the pre-eminent example of this proof-inversion, relying as it does on Littlewood's \lq\lq three principles". His limitation to finite measure spaces helps in this regard; see Chapter 3 in  {\it Real Analysis} (1988), third edition, Macmillan, London. Also, see page 93 of the above text, and the comprehensive treatment of some of these variations, by no means all, in \cite{vy68}. We warn the reader that this sentence is dropped in the fourth edition joint with Fitzpatrick, and a remark added on the assumption of almost everywhere convergence. }   But the basic point  is that there is little more to understand regarding Fatou's lemma as far as real-valued functions are concerned.  As Royden writes, \lq\lq Fatou's lemma and the Monotone Convergence Theorem are very close in the sense that each can be derived from the other using only the fact that  integration is positive and linear." 

As was already understood in the nineteen-seventies,  this is no longer true when one moves from the real line to a finite-dimensional Euclidean space, or more generally, to a Riesz space.  This is simply because the notion of order is weakened and one has to reckon with the fact that the set of cluster points no longer has  a supremum, and the limit function of the sequence with the required properties  has to be found as a selection from a suitable set-theoretic generalization of the notion of an upper limit; see \citet{sc70} and his followers see \cite{hi74,hm71}.  And the situation  becomes even more interesting from a technical point of view when one moves from functions to multifunctions: once the integration of a multifunction is settled upon, a veritable  panoply of measure-theoretic tools is called forth.  In addition to results on measurable selectors, Lyapunov's results on the range of an atomless vector measure begins to attain dominance, as it must, and needs to be invoked.  After all, it is the atomless case that is of consequence;  the purely atomic case, even for $\sigma$-finite measure spaces reduces to the setting of infinite series.  And so what began as a simple cornerstone of the standard theory of integration  has by now attained an architectural multi-facetedness that even when limited to the specific context delimited by the Lyapunov convexity theorem, leads to technical advances that are beyond our scope.  In particular, the recent development that we exclude in this expository article is the consideration of a sequence of multifunctions with unbounded values as in  \cite{bh95,bs05,cmdr05,gp17}. 

One can ask what is in the first place that motivates the consideration of vector-valued functions and multi-functions in the context of Fatou's lemma?  The answer to this question leads into many areas of applied mathematics   including, but not limited to,  optimal control theory, statistical decision theory, the theory of games, and mathematical economics: each field has its own substantive needs and imperatives that lead it to the lemma. In short,    the topic is tremendously vast, and  an answer not only well beyond  the limits of a survey article, but also beyond the competence of the authors writing it.  In this exposition, we limit our focus and attention to Walrasian general equilibrium theory. And here, a rather succinct and sharp answer can be given, one that goes to the very heart of the distinction between an {\it economy} and a {\it game}. 
   
It has been understood since the 1954 paper of Arrow and Debreu that an  economy can be represented as a game by the addition of a fictitious player, and that existence of a Walrasian equilibrium of such an economy deduced as a straightforward consequence of the existence of a Nash equilibrium of the resulting game. But this is less than a full understanding of the problem, and represents only the first step of a two-step argument. A game formalizes interaction between players pursuing conflicting objectives, and a Nash equilibrium takes as given their objectives and choice sets; a Walrasian equilibrium, on the other hand, formalizes, at least in its classical conception that excludes so-called externalities, interaction only between  an individual player and the price system.  This is to say that by necessity the dual space of prices is brought into the picture, and the action or budget sets in which players choose their actions are themselves dependent consequences of their actions. As such, the assumption of given compact action sets, so appropriate  for an analysis of a  game, is just an intermediate auxiliary step for the analysis of an exchange economy where the action set  is itself determined in equilibrium, and given the eminent possibility of some commodities having zero prices in equilibrium, is certainly not compact.  The coordination of actions has to be brought out by aggregation of these actions, and then furthermore, on two additional considerations to bound these actions: (i) on the reliance of the finiteness of aggregate resources  of the economy, and (ii) on the concurrent impossibility of their becoming infinite as a consequence of the technological possibilities available to the economy, to bound the actions.  These considerations are missing in a game by definition.   And it is precisely these considerations that require a second step whereby the arbitrary compactifications introduced in the first Arrow--Debreu step have to be supplemented by an argument that emphasizes the careful choice of these compactifications and  establishes their consequent irrelevancy.  
   
For economies with a continuum of agents, it is this second step that  necessitates Fatou's lemma;  and when each agent in a continuum of agents is made to  choose from a budget set pertaining to a continuum of commodities, it necessitates a conception of aggregation as integration, and thereby necessitates      a generalization of the lemma to integrable multifunctions taking values  in infinite-dimensional commodity spaces. The idealized equilibrium for the idealized economy is thus determined as the limit of  the  equilibria of a  truncated or perturbed economy, to be sure also an idealization in another register, but one that is manageable as a game. It is as if the idealized shape of the object at hand is determined by studying the trembles of that shape, which is to say, by focusing on a  sequence of functions that have a well-behaved limit for the idealization. If there is any originality in this expository survey of recent developments of Fatou's lemma as applied to Walrasian equilibrium theory, it lies in our giving rather salient prominence to these  truncations and perturbations  --- trembles if one wishes --- in their formalization as Galerkin approximations. Such approximations have been implicitly used ever since the initiation of the subject in the work of Bewley and his followers. 
   
The structure of this survey is then rather straightforward. We begin with a quick view of the technical development of Fatou's lemma in a finite-dimensional Euclidean space  and its application to finite-dimensional Walrasian   general equilibrium theory in conjunction with the Lyapunov convexity theorem. In Section 3, we turn to the Banach space setting, and provide for the convenience of the reader a self-contained treatment of the relevant measure-theoretic notions:   measurable multifunctions and their measurable selectors, integration of multifunctions, and most importantly in terms of the recent developments in the subject, the role of the so-called saturation property in Lyapunov's convexity theorem.  It is the latter that leads us to distinguish between the approximate and exact versions of Fatou's lemma in the context of the two prevailing notions of integration. This distinction between the exact and the approximate is a different kind of tremble, one that pertains to the lemma itself rather than to its application.  After a brisk treatment of Galerkin approximation, we turn to the application in Section 7 and present our main findings formalized as Theorems \ref{WE1} and \ref{WE2} pertaining to infinite-dimensional Walrasian general equilibrium theory.  We conclude the survey by picking up again some of the issues with which we begin this  introduction.

\section{Background of the Problem}
This  section introduces the reader to the initial motivation and pioneering efforts that led to the development of the subject  in the context of a  finite-dimensional Euclidean space.

\subsection{Fatou's Lemma in Finite Dimensions}
\label{fat1}
For any vectors $x$ and $y$ in $\R^k$, where $x^i$ and $y^i$ are their $i$th components respectively, define the vector inequalities by: 
\begin{align*}
& x\ge y \Longleftrightarrow x^i\ge y^i,\  i=1,\dots,k; \\
& x>y \Longleftrightarrow x\ge y \text{ and }x\ne y; \\
& x\gg y \Longleftrightarrow x^i>y^i,\ i=1,\dots,k.
\end{align*}
The nonnegative orthant of $\R^k$ is denoted by $\R^k_+=\{x\in \R^k\mid x\ge 0 \}$ and the positive orthant of $\R^k$ is denoted by $\R^k_{++}=\{x\in \R^k\mid x\gg 0 \}$. 

The \textit{upper limit} of a sequence $\{ F_n \}_{n\in \N}$ of subsets in $\R^k$ is defined by
$$
\mathrm{Ls}\left\{ F_n \right\}=\left\{ x\in E \mid \exists \left\{ x_{n_i} \right\}_{i\in \N}: x=\lim_{i\to \infty}x_{n_i},\,x_{n_i}\in F_{n_i}\,\forall i\in \mathbb{N} \right\},
$$
where $\{ x_{n_i} \}_{i\in \N}$ denotes a subsequence of $\{ x_n \}_{n\in \N}\subset \R^k$. 

Let $(T, \Sigma,\mu)$ be a measure space. In the rest of the essay $(T, \Sigma,\mu)$ is always assumed to be complete without any explicit mention. Denote by $L^1(\mu)$ the space of ($\mu$-equivalence classes) of integrable functions on $T$ and by $L^\infty(\mu)$ the space of ($\mu$-equivalence classes) of essentially bounded measurable  functions on $T$. The space of ($\mu$-equivalence classes) of integrable functions from $T$ to $\R^k$, the $k$-dimensional Euclidean space, is denoted by $L^1(\mu,\R^k)$. A set $A\in \Sigma$ with $\mu(A)>0$ is called an \textit{atom} of a measure $\mu$ if $B\subset A$ with $B\in \Sigma$ implies $\mu(A)=\mu(B)$ or $\mu(B)=0$. A measure is said to be a \textit{nonatomic} if it possesses no atom. The classical Lyapunov convexity theorem says that the range $m(\Sigma)$ of a vector measure $m:\Sigma \to \R^k$ is compact and convex whenever its component measure $m^i:\Sigma\to \R$ with $m=(m^1,\dots,m^k)$ is a nonatomic finite signed measure for each $i=1,\dots,k$ (see \cite[Corollary IX.1.5]{du77} or \cite[Theorem 5.5]{ru91}).  

Curiously, the first Fatou's lemma in finite dimensions appeared in the multifunction case in the work of \cite{au65}. To state the result, denote by $\M^1(\mu,\R^k)$ the set of integrably bounded multifunctions from $T$ to $\R^k$. 

\begin{thm}[\citet{au65}]
\label{au1}
Let $(T,\Sigma,\mu)$ be a nonatomic finite measure space. If $\{ \Gamma_n \}_{n\in \N}$ is a well-dominated sequence\footnote{See Subsection \ref{bo} for the definition of \textit{well dominance}.} of multifunctions in $\M^1(\mu,\R^k)$, then: 
$$
\mathrm{Ls}\left\{ \int \Gamma_nd\mu \right\}\subset \int \mathrm{Ls}\left\{ \Gamma_n \right\}d\mu. 
$$
\end{thm}

\noindent The specialization to the (single-valued) function case  first appeared in \cite{sc70} in  the following form. 

\begin{thm}[\citet{sc70}]
\label{sc}
Let $(T,\Sigma,\mu)$ be a nonatomic finite measure space. If $\{ f_n \}_{n\in \N}$ is a sequence of integrable functions from $T$ to $\R^k_+$ for which $\lim_n\int f_nd\mu$ exists, then there exists an integrable function $f:T\to \R^k_+$ with the following properties. 
\begin{enumerate}[\rm (i)]
\item $f(t)\in \mathrm{Ls}\left\{ f_n(t) \right\}$ a.e. $t\in T$.
\item $\displaystyle\int fd\mu\le \lim_{n\to \infty}\int f_nd\mu$.
\end{enumerate}
\end{thm}

\noindent As remarked in \cite[p.\,300]{sc70}: 
\begin{quote}
When $k=1$, the result is a form of Fatou's lemma.\ [...]\ The nontrivial part of the arguments is limited to the case where $T$ is an atomless measure space. In the purely atomic case [Theorem \ref{sc}] is reduced to a simple exercise in series. In any case, [Theorem \ref{sc}] cannot be proved by a successive application of Fatou's lemma $k$ times.
\end{quote}
This observation is based on the fact that the number of atoms of $\mu$ is countable and any measurable function on $T$ is constant on each atom, so that for the nonatomic parts of $T$, the Lyapunov convexity theorem plays a crucial role. It should be noted that the inclusion form at the limit is a new aspect that is entirely absent from the classical Fatou's lemma. 

Meanwhile, an alternative proof of Theorem \ref{sc} appeared in \cite{hm71}, in which the following open question was raised \citep[Remark]{hm71}: 

\begin{quote}
If the sequence $\{ f_n \}_{n\in \N}$ majorized by an integrable function, then there is a function $f$ such that a.e.\ in $T$, $f(t)\in \mathrm{Ls}\,\{ f_n(t) \}$ and $\int fd\mu=\lim_n\int f_nd\mu$. Does this still hold when $\{ f_n \}_{n\in \N}$ is uniformly integrable?
\end{quote}
The affirmative answer was given in the following result. 

\begin{thm}[\citet{ar79}]
\label{ar1}
Let $(T,\Sigma,\mu)$ be a finite measure space. If $\{ f_n \}_{n\in \N}$ is a uniformly integrable sequence in $L^1(\mu,\R^k)$ for which $\lim_n\int f_nd\mu$ exists, then there exists $f\in L^1(\mu,\R^k)$ with the following properties. 
\begin{enumerate}[\rm (i)]
\item $f(t)\in \mathrm{Ls}\left\{ f_n(t) \right\}$ a.e. $t\in T$;
\item $\displaystyle\int fd\mu=\lim_{n\to \infty}\int f_nd\mu$.
\end{enumerate}
\end{thm}

The significance of \cite{ar79} does not lie in the solution to the open question per se, but in the eduction of the following result in the course of the proof. It is this  that leads to a variant of Fatou's lemma in  the so-called \lq\lq inclusion form."

\begin{lem}[\citet{ar79,hm71}]
\label{lem1}
Let $(T,\Sigma,\mu)$ be a finite measure space. If $\{ f_n \}_{n\in \N}$ is a uniformly integrable sequence in $L^1(\mu,\R^k)$ such that $f_n\to f$ weakly in $L^1(\mu,\R^k)$, then: 
$$
f(t)\in \mathrm{co}\,\mathrm{Ls}\left\{ f_n(t) \right\} \quad\text{a.e. $t\in T$}.
$$
\end{lem}
\noindent
In this essay we focus on the following prototype of Fatou's lemma in finite dimensions, in which the order structure is completely replaced by inclusions. To crystallize the usefulness of Lemma \ref{lem1} and illustrate the role of nonatomicity in the subject, we provide a proof explored in \cite{km86,ks14a,ya89} for completeness.   

\begin{thm}
\label{thm1}
Let $(T,\Sigma,\mu)$ be a nonatomic finite measure space. If $\{ f_n \}_{n\in \N}$ is a uniformly integrable, bounded  sequence in $L^1(\mu,\R^k)$, then:
\begin{enumerate}[\rm (i)]
\item $\mathrm{Ls}\left\{ \displaystyle\int f_nd\mu \right\}\subset \displaystyle\int\mathrm{Ls}\left\{ f_n \right\}d\mu$. 
\item There exists $f\in L^1(\mu,\R^k)$ such that
\begin{enumerate}[\rm (a)]
\item $f(t)\in \mathrm{Ls}\left\{ f_n(t) \right\}$ a.e.\ $t\in T$;
\item $\displaystyle\int fd\mu\in \mathrm{Ls}\left\{ \int f_nd\mu \right\}$.
\end{enumerate}
\end{enumerate}
\end{thm}

\begin{proof}
Since the uniformly integrable, bounded sequence $\{ f_n \}_{n\in \N}$ is relatively weakly compact in $L^1(\mu,\R^k)$ by the Dunford--Pettis criterion (see \cite[Corollary IV.8.11]{ds58}), one can extract from it a subsequence $\{ f_{n_i} \}_{i\in \N}$ that converges weakly to $f_0$ in $L^1(\mu,\R^k)$. Since the integration operator $g\mapsto \int gd\mu$ from $L^1(\mu,\R^k)$ to $\R^k$ is continuous in the weak topology for $L^1(\mu,\R^k)$, we obtain $\int f_id\mu\to \int f_0d\mu$, and hence, $\int f_0d\mu\in \mathrm{Ls}\,\{ \int f_nd\mu \}$. It follows from Lemma \ref{lem1} that $f_0(t)\in \overline{\mathrm{co}}\,\mathrm{Ls}\,\{ f_n(t) \}$ a.e.\ $t\in T$. Integrating the both sides of this inclusion yields $\int f_0d\mu\in \int\overline{\mathrm{co}}\,\mathrm{Ls}\,\{ f_n \}d\mu=\int\mathrm{Ls}\,\{ f_n \}d\mu$, where the equality follows from the classical Lyapunov theorem (see \cite[Theorem 8.6.3]{af90} or \cite[Theorem 3 and Proposition 7, D.II.4]{hi74}). Hence, there exists $f\in L^1(\mu,\R^k)$ such that $f(t)\in \mathrm{Ls}\,\{ f_n(t) \}$ a.e.\ $t\in T$ and $\int fd\mu=\int f_0d\mu\in \mathrm{Ls}\,\{ \int f_nd\mu \}$, which verifies condition (ii). 

Condition (i) follows easily from condition (ii). To see this claim, take any $x\in \mathrm{Ls}\,\{ \int f_nd\mu \}$. Then there exists a subsequence $\{ f_{n_i} \}_{i\in \N}$ such that $\lim_i\int f_{n_i}d\mu=x$. It follows from condition (ii) that there exists $f\in L^1(\mu,\R^k)$ such that (a) $f(t)\in \mathrm{Ls}\,\{ f_{n_i}(t) \}$ a.e.\ $t\in T$; (b) $\int fd\mu=\lim_i\int f_{n_i}d\mu$. Integrating the both side of condition (a) together with condition (b) yields $x\in \int\mathrm{Ls}\,\{ f_{n_i} \}d\mu\subset \int\mathrm{Ls}\,\{ f_n \}d\mu$.
\end{proof}

For  later reference, given a sequence $\{ f_n \}_{n\in \N}$ in $L^1(\mu,\R^k)$ such that $\mathrm{Ls}\left\{ \int f_nd\mu \right\}$ is nonempty, condition (i) of Theorem \ref{thm1} is referred to the \textit{Fatou property} and condition (ii) of Theorem \ref{thm1} is referred to the \textit{upper closure property}.\footnote{We borrow this terminology from \citet{ce83}, who called the Fatou-like lemma the ``lower closure theorem'', which is most useful in proving the existence of solutions in optimal control problems. For the relevance of \citet{sc70} to the lower closure theorem, see \cite{cs78}.} Our attempt in this essay  is to exemplify how these notions characterize the structure of nonatomic finite measure spaces. For a further extension and another variant of the Fatou and the upper closure properties in the finite-dimensional case, see \cite{ba84,cs78,ol87}. 

To show how Fatou's lemma for multifunctions easily follows from that for functions, we provide here a proof of Theorem \ref{au1} exploiting Theorem \ref{thm1} for completeness. 

\begin{proof}[Proof of Theorem \ref{au1}]
If $\mathrm{Ls}\,\{ \int\Gamma_n d\mu \}=\emptyset$, then the result is trivially true. Thus, without loss of generality, we may assume that $\mathrm{Ls}\,\{ \int\Gamma_n d\mu \}\ne\emptyset$. Take any $x\in \mathrm{Ls}\,\{ \int \Gamma_nd\mu \}$. Then there is a sequence $\{ x_n \}_{n\in \N}$ in $\R^k$ with $x_n\in \int \Gamma_nd\mu$ one can extract a subsequence $\{ x_{n_i} \}_{i\in \N}$ such that $x_{n_i}\to x$. Hence, there is an integrably bounded sequence $\{ f_{n_i} \}_{i\in \N}$ in $L^1(\mu,\R^k)$ such that $f_{n_i}$ is a measurable selector of $\Gamma_{n_i}$ for each $i$ and $x_{n_i}=\int f_{n_i}d\mu$. It follows from Theorem \ref{thm1} that
$$
x=\lim_{i\to \infty}x_{n_i}\in \mathrm{Ls}\left\{ \int f_nd\mu \right\}\subset \int \mathrm{Ls}\left\{ f_n \right\}d\mu\subset \int \mathrm{Ls}\left\{ \Gamma_n \right\}d\mu.
$$
Therefore, the desired inclusion holds.
\end{proof}

\subsection{Large Economies with Finite-Dimensional Commodity Spaces}
\label{LE1}
Under the standard setting of large economies with finite-dimensional commodity spaces, the consumption set of each agent is assumed to be an unbounded subset of $\R^k_+$ and the preferences of each agent are supposed to be monotonic or locally insatiated (see \cite{au66,hi70a,hi74,ky81,sc69}). These assumptions cause a serious difficulty in the integrability (summability) of the demand multifunction and the detection of Walrasian equilibria in large economies, which is a peculiar problem stemming from the fact that the agent space is a continuum. To overcome this difficulty, two approaches to the direct application of Fatou's lemma were proposed. 
 
\begin{enumerate}[(1)]
\item The truncation of the consumption set (see \cite{hi70a,sc69}): In this approach, the consumption set of each agent is truncated into a subset that is bounded from above. The individual demand is restricted to the truncated consumption set and the upper bound of the consumption set ensures the integrable boundedness of the truncated demand multifunction. 
\item The truncation of the price simplex (see \cite{ky81}): In this approach, the price simplex is truncated into a subset that excludes a small neighborhood of the origin. Since the available prices are away from the zero price, the demand multifunction restricted to the set of truncated prices is integrably bounded.   
\end{enumerate}
In both approaches, the underlying truncation furnishes a sequence of truncated subeconomies with the original economy, in which truncated equilibria exists. In the first approach, Fatou's lemma arises in the limit argument when the upper bound of the consumption set tends to infinity and in the second approach, so does it when the neighborhood of the origin shrinks gradually to the singleton. Following the first approach, we illustrate in the sequel  the power of Fatou's lemma in Walrasian general equilibrium analysis.  

The set of agents is given by a complete finite measure space $(T,\Sigma,\mu)$. There exists $k\in \N$ distinct commodities available and the commodity space is given by $\R^k$. The preference relation ${\succsim}(t)$ of each agent $t\in T$ is a complete, transitive binary relation on a common consumption set $\R^k_+$, which induces the preference mapping $t\mapsto {\succsim}(t)\subset \R^k_+\times \R^k_+$. We denote by $x\,{\succsim}(t)\,y$ the relation $(x,y)\in {\succsim}(t)$. The indifference and strict relations are defined respectively by $x\,{\sim}(t)\,y$ $\Leftrightarrow$ $x\,{\succsim}(t)\,y$ and $y\,{\succsim}(t)\,x$, and by $x\,{\succ}(t)\,y$ $\Leftrightarrow$ $x\,{\succsim}(t)\,y$ and $x\,{\not\sim}(t)\,y$. Each agent possesses an initial endowment $\omega(t)\in \R^k_+$, which is the value of an integrable function $\omega:T\to \R^k$. The economy $\E$ consists of the primitives $\E=\{ (T,\Sigma,\mu),\R^k_+,\succsim,\omega \}$. 

The price space is $\R^k$. Given a price $p\in \R^k\setminus \{ 0 \}$, the budget set of each agent is $B(t,p)=\{ x\in \R^k_+\mid \langle p,x \rangle \le \langle p,\omega(t) \rangle \}$. A function $f\in L^1(\mu,\R^k)$ is called an \textit{assignment} if $f(t)\in \R^k_+$ a.e.\ $t\in T$. An assignment $f$ is an \textit{allocation} for $\E$ if $\int fd\mu=\int \omega d\mu$. 

\begin{dfn}
A price-allocation pair $(p,f)$ is a \textit{Walrasian equilibrium} for $\E$ if a.e.\ $t\in T$: $f(t)\in B(t,p)$ and $x\not\in B(t,p)$ whenever $x\,{\succ}(t)\,f(t)$.
\end{dfn}

The standing assumption on $\E$ is described as follows. 

\begin{assmp}
\label{assmp1}
\begin{enumerate}[(i)]
\item ${\succsim}(t)$ is a closed subset of $\R^k_+\times \R^k_+$ for every $t\in T$. 
\item For every assignment $f$ and $g$: the set $\{ t\in T\mid f(t)\,{\succsim}(t)\,g(t) \}$ belongs to $\Sigma$. 
\item For every $x,y\in \R^k_+$: $x>y$ implies $x\,{\succ}(t)\,y$. 
\item $\displaystyle\int\omega d\mu\in \R^k_{++}$. 
\end{enumerate}
\end{assmp}

Condition (i) is referred to as  \textit{continuity} of the preference relation ${\succsim}(t)$. This is equivalent to the requirement that for every $x\in \R^k_+$ both the upper contour set $\{ y\in \R^k_+\mid y\,{\succsim}(t)\,x \}$ and the lower contour set $\{ y\in \R^k_+\mid x\,{\succsim}(t)\,y \}$ are closed in $\R^k_+$. The measurability of the preference mapping in condition (ii) is introduced in \cite{au66}, and condition (iii) is referred to as  \textit{monotonicity} of ${\succsim}(t)$. This means that every commodity is desirable to each agent. Condition (iv) ``asserts that no commodity is totally absent from the market \cite[p.\,3]{au66}''. It follows from the continuity and the monotonicity of ${\succsim}(t)$ that:
\begin{description}
\item[($\star$)] For every $x\in \R^k_+$: $x$ belongs to the closure of the upper contour set $\{ y\in \R^k_+\mid y\,{\succ}(t)\,x \}$. 
\end{description}
Condition ($\star$) is a variant of ``local nonsatiation'' that originated in \cite{hi68}: it  plays a crucial role in proving the existence of Walrasian equilibria with free disposal without the monotonicity of preferences in Section \ref{sec}.  

\begin{thm}[\citet{au66}]
\label{au2}
Let $(T,\Sigma,\mu)$ be a nonatomic finite measure space. Then for every economy $\E$ satisfying Assumption \ref{assmp1}, there exists a Walrasian equilibrium $(p,f)$ with $p\in \R^k_{++}$.
\end{thm}

Intuitively, the idea of \citeauthor{au66}'s proof is divided into two steps: first, the unbounded consumption set of each agent is truncated into a bounded set and equilibria are detected in the truncated subeconomy. Towards  this end, the Lyapunov convexity theorem guarantees that the ``aggregate preferred multifunction'' has compact, convex values (see \cite{au65}) and Fatou's lemma for multifunctions (Theorem \ref{au1}) guarantees that it is upper semicontinuous in prices (see \cite{au76}). This crucial observation makes it possible for him  to apply Brouwer's fixed point theorem in the truncated subeconomy. Second, for each sequence of the upper bounds of the consumption set, there exists a sequence of equilibria in the truncated subeconomies. Letting the upper bounds tend to infinity yields that the limit of the sequence of equilibria corresponds to a Walrasian equilibrium in the original economy, though the formal argument is somewhat technically intricate.\fn{This, and the succeeding paragraph, formalize  considerations already  informally discussed in the introduction, but the reader should note that in this essay  \lq\lq trembles"  are also relevant to an approximate Fatou's lemma with $\varepsilon$-perturbations in Subsections \ref{app1} and 5.1.} 

It is \citet{sc69} who facilitated the limiting argument in the second step using Fatou's lemma for functions (Theorem \ref{sc}). We outline here the proof of Theorem \ref{au2} based on Theorem \ref{ar1} instead of Theorem \ref{sc} for the sake of simplicity.\footnote{The proof is, however, rather different from \cite{sc69}, especially in Steps 3 and 4 below. See \cite{ya01} for another heuristic illustration based on the argument in \cite{hm71} employing Theorem \ref{sc}.}\ To this end, let $e=(1,\dots,1)\in \R^k$ and define the $n$-bounded partial budget set by 
$$
B^n(t,p)=B(t,p)\cap\left\{ x\in \R^k_+\mid n\left( 1+\sum_{i=1}^k\omega^i(t)e \right) \right\},
$$
where $\omega^i(t)$ is the $i$th component of $\omega(t)\in \R^k_+$. A price-allocation pair $(p,f)$ is called an \textit{$n$-bounded partial Walrasian equilibrium} for $\E$ if a.e.\ $t\in T$: $f(t)\in B^n(t,p)$ and $x\not\in B^n(t,p)$ whenever $x\,{\succ}(t)\,f(t)$.

\begin{aux}[\citet{sc69}]
Let $(T,\Sigma,\mu)$ be a nonatomic finite measure space. If the economy $\E$ satisfies Assumption \ref{assmp1}, then for each $n\in \N$ there exists an $n$-bounded partial Walrasian equilibrium $(p,f)$ with $p\in \R^k_+\setminus \{ 0 \}$.
\end{aux}

\begin{proof}[Proof of Theorem \ref{au2}]
\underline{Step 1}: By Auxiarily Theorem, there exists an $n$-bounded partial Walrasian equilibrium $(q_n,f_n)\in (\R^k_+\setminus \{ 0 \})\times L^1(\mu,\R^k)$ for $\E$. We can normalize the $n$-bounded partial equilibrium price for $\E$ such that $p_n=q_n/\| q_n \|\in \Delta:=\{ p\in \R^k_+\mid \| p \|=1 \}$. Since the price simplex $\Delta$ is compact, we can extract a subsequence from $\{ p_n \}_{n\in \N}$ (which we do not relabel) that converges to $p\in \Delta$. Via a somewhat intricate argument using the monotonicity of ${\succsim}(t)$ in Assumption \ref{assmp1}(iii), we can also show that $p\in \R^k_{++}$ (see \cite[Main Lemma]{sc69}). This is a crucial point to guarantee the compactness of the budget set $B(t,p)$ (see also \cite[Lemma 6.1]{au66}). 

\underline{Step 2}: Since $\int f_nd\mu=\int\omega d\mu$ for each $n\in \N$, it is obvious that $\{ f_n \}_{n\in \N}$ is a uniformly integrable sequence with $\lim_n\int f_nd\mu=\int\omega d\mu$. Invoking Theorem \ref{ar1} yields that there exist an assignment $f\in L^1(\mu,\R^k)$ such that $f(t)\in \mathrm{Ls}\,\{ f_n(t) \}$ a.e.\ $t\in T$ and $\int fd\mu=\lim_{n}\int f_nd\mu=\int\omega_n d\mu$. Therefore, $f$ is an allocation for $\E$. 

\underline{Step 3}: 
We claim that: 
\begin{quote}
For a.e.\ $t\in T$: $x\,{\succ}(t)\,f(t)$ implies that $\langle p,x \rangle>\langle p,\omega(t) \rangle$.
\end{quote}
Suppose that $\langle p,\omega(t) \rangle=0$. If $\langle p,x \rangle=0$, then $x=0$ in view of $p\in \R^k_{++}$, but $x\,{\succ}(t)\,f(t)$ is impossible by the monotonicity of ${\succsim}(t)$ in Assumption \ref{assmp1}(iii). Hence, the claim is automatic whenever $\langle p,\omega(t) \rangle=0$. Suppose, by way of contradiction, that there exists $A\in \Sigma$ of positive measure with the following property. For every $t\in A$: $\langle p,\omega(t) \rangle>0$ and there exists $y\in \R^k$ such that $y\,{\succ}(t)\,f(t)$ and $\langle p,y \rangle\le \langle p,\omega(t) \rangle$. It follows from the continuity of ${\succsim}(t)$ that $\varepsilon y\,{\succ}(t)\,f(t)$ and $\langle p,\varepsilon y \rangle<\langle p,\omega(t) \rangle$ for some $\varepsilon\in (0,1)$. Hence, we may assume without loss of generality that for every $t\in A$ there exists $y\in \R^k$ such that $y\,{\succ}(t)\,f(t)$ and $\langle p,y \rangle<\langle p,\omega(t) \rangle$. Define the multifunction $\Gamma:A\twoheadrightarrow \R^k$ by 
$$
\Gamma(t)=\left\{ x\in \R^k\mid x\,{\succ}(t)\,f(t),\ \langle p,x \rangle<\langle p,\omega(t) \rangle \right\}.
$$
Then $\Gamma$ is a graph measurable multifunction with $y\in \Gamma(t)$.\footnote{We shall show the graph measurability of $\Gamma$ as a consequence of a  more general result pertaining to  the setting of a  separable Banach space in the course of the proof of Theorem \ref{WE1}.} Let $h:A\to \R^k$ be a measurable selector from $\Gamma$. Suppose that the set defined by
$$
\bigcup_{n\in \N}\{ t\in A\mid h(t)\,{\succ}(t)\,f_n(t), \,\langle p_n,h(t)\rangle<\langle p_n,\omega(t) \rangle \}
$$
is of measure zero. Then for each $n\in \N$: $f_n(t)\,{\succsim}(t)\,h(t)$ or $\langle p_n,h(t)\rangle\ge \langle p_n,\omega(t) \rangle$ a.e.\ $t\in A$. Since $p_n\to p$, passing to the limit along a suitable subsequence of $\{ f_n(t) \}_{n\in \N}$ in $\R^k$ yields $f(t)\,{\succsim}(t)\,h(t)$ or $\langle p,h(t)\rangle\ge \langle p,\omega(t) \rangle$ a.e.\ $t\in A$, a contradiction to the fact that $h$ is a measurable selector from $\Gamma$. Therefore, there exists $n\in  \N$ such that $\{ t\in A\mid h(t)\,{\succ}_n(t)\,f_n(t), \,\langle p_n,h(t)\rangle<\langle p_n,\omega(t) \rangle \}$ is of positive measure, and hence, $h(t)\in B^n(t,p_n)$ for every $t$ in this set of positive measure. This is, however, impossible because $(p_n,f_n)$ is an $n$-bounded partial Walrasian equilibrium for $\E$. Therefore, the claim is true. 

\underline{Step 4}: It remains to show that $\langle p,f(t) \rangle\le \langle p,\omega(t) \rangle$ a.e.\ $t\in T$. Since $f(t)$ belongs to the closure of the upper contour set $\{  x\in \R^k_+\mid x\,{\succ}(t)\,f(t) \}$ by condition ($\star$), the claim shown in Step 3 implies that $\langle p,f(t) \rangle\ge \langle p,\omega(t) \rangle$. Integrating  both sides of this inequality yields $\int \langle p,f(t) \rangle d\mu\ge \int \langle p,\omega(t) \rangle d\mu$. On the other hand, as demonstrated in Step 2, $\int fd\mu=\int \omega d\mu$, and hence, $\int  \langle p,f(t) \rangle d\mu=\int \langle p,\omega(t) \rangle d\mu$. Therefore, we must have the equality $\langle p,f(t) \rangle=\langle p,\omega(t) \rangle$ a.e.\ $t\in T$. Therefore, $(p,f)\in \R^k_{++}\times L^1(\mu,\R^k)$ is a Walrasian equilibrium for $\E$. 
\end{proof}

\section{Vector Integration in Banach Spaces}
In this section we first provide an overview of the two standard formulations of integration in an infinite-dimensional space: Bochner and Gelfand integrals of functions, measurable selectors of multifunctions with values in a Banach space or its dual, the corresponding integrals of multifunctions, and some additional terminologies from vector integration theory. Then we introduce the notion of saturated measure space and provide its complete characterization in terms of the Lyapunov convexity theorem 
and the compactness and convexity of the corresponding  integrals of multifunctions.

\subsection{Bochner Integrals of Multifunctions}
\label{bo}
Let $(T,\Sigma,\mu)$ be a (complete) finite measure space. Let $(E,\|\cdot\|)$ be a Banach space with its dual $E^*$ furnished with the dual system $\langle \cdot,\cdot \rangle$ on $E\times E^*$. A function $f:T\to E$ is \textit{strongly measurable} if there exists a sequence of simple (or finitely valued) functions $f_n:T\to E$ such that $\|f(t)-f_n(t)\|\to 0$ a.e.\ $t\in T$; $f$ is \textit{Bochner integrable} if it is strongly measurable and $\int \|f(t)\|d\mu<\infty$, where the \textit{Bochner integral} of $f$ over $A\in \Sigma$ is defined by $\int_A fd\mu=\lim_n\int_A f_nd\mu$. Denote by $L^1(\mu,E)$ the space of ($\mu$-equivalence classes of) $E$-\hspace{0pt}valued Bochner integrable functions on $T$, normed by $\| f \|_1=\int \| f(t) \|d\mu$, $f\in L^1(\mu,E)$. By the Pettis measurability  theorem (see \cite[Theorem II.1.2]{du77}), $f$ is strongly measurable if and only if it is Borel measurable with respect to the norm topology of $E$ whenever $E$ is separable.  

A function $g:T\to E^*$ is \textit{weakly$^*\!$ scalarly measurable} if for every $x\in E$ the scalar function $\langle x,g(\cdot) \rangle:T\to \R$ defined by $t\mapsto \langle x,g(t) \rangle$ is measurable. Denote by $L^\infty(\mu,E^*_{\textit{w}^*})$ the space of ($\mu$-equivalence classes of) weakly$^*\!$ measurable, essentially bounded, $E^*$-valued functions on $T$, normed by $\| g \|_\infty=\mathrm{ess\,sup}_{t\in T}\| g(t) \|<\infty$. Then the dual space of $L^1(\mu,E)$ is given by $L^\infty(\mu,E^*_{\textit{w}^*})$ whenever $E$ is separable (see \cite[Theorem 2.112]{fl07} or \cite[Corollary to Theorem VII.4.7]{it69}), and the dual system is given by $\langle f,g \rangle=\int\langle f(t),g(t) \rangle d\mu$ with $f\in L^1(\mu,E)$ and $g\in L^\infty(\mu,E^*_{\textit{w}^*})$. Denote by $\mathrm{Borel}(E^*,\mathit{w}^*)$ the Borel $\sigma$-\hspace{0pt}algebra of $E^*$ generated by the weak$^*\!$ topology. If $E$ is a separable Banach space, then $E^*$ is a locally convex Suslin space under the weak$^*\!$ topology (see \cite[p.\,67]{th75}). Hence, under the separability of $E$, a function $g:T\to E^*$ is weakly$^*\!$ scalarly measurable if and only if it is Borel measurable with respect to $\mathrm{Borel}(E^*,\mathit{w}^*)$ (see \cite[Theorem 1]{th75}). Weakly$^*\!$ scalarly measurable functions $g_1,g_2:T\to E^*$ are \textit{weakly$^*\!$ scalarly equivalent} if $\langle x,g_1(t) \rangle=\langle x,g_2(t) \rangle$ for every $x\in E$ a.e.\ $t\in T$ (the exceptional $\mu$-\hspace{0pt}null set depending on $x$). 

A set-\hspace{0pt}valued mapping from $T$ to the family of nonempty subsets of $E$ is called a \textit{multifunction}. A multifunction $\Gamma:T\twoheadrightarrow E$ is \textit{measurable} if the set $\{t\in T\mid \Gamma(t)\cap U\ne \emptyset \}$ is in $\Sigma$ for every open subset $U$ of $E$; it is \textit{graph measurable} if the set $\mathrm{gph}\,\Gamma:=\{ (t,x)\in T\times E\mid x\in \Gamma(t) \}$ belongs to $\Sigma\otimes \mathrm{Borel}(E,\| \cdot \|)$, where $\mathrm{Borel}(E,\| \cdot \|)$ is the Borel $\sigma$-\hspace{0pt}algebra of $(E,\| \cdot \|)$ generated by the norm topology. If $E$ is separable, then $\mathrm{Borel}(E,\| \cdot \|)$ coincides with the Borel $\sigma$-\hspace{0pt}algebra $\mathrm{Borel}(E,\mathit{w})$ of $E$ generated by the weak topology (see \cite[Corollary 2, Part I, Chap.\,II]{sc73} or \cite[p.\,21]{he90}). It is well-\hspace{0pt}known that for closed-\hspace{0pt}valued multifunctions, measurability and graph measurability coincide whenever $E$ is separable (see \cite[Theorem 8.1.4]{af90} or \cite[Theorem III.30]{cv77}). A function $f:T\to E$ is a \textit{selector} of $\Gamma$ if $f(t)\in \Gamma(t)$ a.e.\ $t\in T$. If $E$ is separable, then by the Aumann measurable selection theorem, a multifunction $\Gamma$ with measurable graph admits a measurable selector (see \cite[Theorem III.22]{cv77} or \cite[Theorem 1, D.II.2]{hi74}) and it is also strongly measurable. 

Let $B$ be the open unit ball in $E$. A multifunction $\Gamma:T\twoheadrightarrow E$ is \textit{integrably bounded} if there exists $\varphi\in L^1(\mu)$ such that $\Gamma(t)\subset \varphi(t)B$ a.e.\ $t\in T$. If $\Gamma$ is graph measurable and integrably bounded, then it admits a Bochner integrable selector whenever $E$ is separable. Denote by $\mathcal{S}^1_\Gamma$ the set of Bochner integrable selectors of $\Gamma$. If $\Gamma$ is an integrably bounded, measurable multifunction with weakly compact, convex values, then $\mathcal{S}^1_\Gamma$ is weakly compact in $L^1(\mu,E)$ whenever $E$ is separable (see \cite[Theorem 3.1]{ya91}). The Bochner integral of $\Gamma$ is conventionally defined as $\int\Gamma d\mu:=\{ \int fd\mu \mid f\in \mathcal{S}^1_\Gamma \}$. Denote by $\overline{\mathrm{co}}\,\Gamma$ the multifunction defined by the closure of the convex hull of $\Gamma(t)$. 

A subset $K$ of $L^1(\mu,E)$ is said to be \textit{uniformly integrable} if 
$$
\lim_{\mu(A)\to 0}\sup_{f\in K}\int_A\| f \|d\mu=0. 
$$
$K$ is said to be \textit{well-\hspace{0pt}dominated} if there is an integrably bounded, weakly compact-\hspace{0pt}valued multifunction $\Gamma:T\twoheadrightarrow E$ such that $f(t)\in \Gamma(t)$ a.e.\ $t\in T$ for every  $f\in K$. Here, $\Gamma$ is referred to as a \textit{dominating multifunction} for $K$. Denote by $\M^1(\mu,E)$ the set of integrably bounded multifunctions from $T$ to $E$. A subset $\K$ of $\M^1(\mu,E)$ is said to be \textit{well-dominated} if there exists an integrably bounded, weakly compact-valued multifunction $\tilde{\Gamma}:T\twoheadrightarrow E$ such that $\Gamma(t)\subset \tilde{\Gamma}(t)$ for every $\Gamma\in \M^1(\mu,E)$ and $t\in T$. 

\begin{thm}[\citet{di77,drs93}]
\label{diest}
Let $(T,\Sigma,\mu)$ be a finite measure space and $E$ be a Banach space. Then a well-dominated subset of $L^1(\mu,E)$ is relatively weakly compact. 
\end{thm}

Well-dominance provides a Bochner integral analogue of the Dunford--Pettis criterion for the relative weak compactness in $L^1(\mu,E)$. For the development of the weak compactness in $L^1(\mu,E)$, see also  \cite{di96,du77,kh84a,ta84,ug91}.

\subsection{Gelfand Integrals of Multifunctions}
\label{gel}
A weakly$^*$ scalarly measurable function $f:T\to E^*$ is \textit{Gelfand integrable} over $A\in \Sigma$ if there exists $x^*_A\in E^*$ such that $\langle x,x^*_A \rangle=\int_A\langle x,f(t) \rangle d\mu$ for every $x\in E$. The element $x^*_A$, which is unique by the separation theorem, is called the \textit{Gelfand integral} (or \textit{weak$^*\!$ integral}) of $f$ over $A$, and is denoted by $\mathit{w}^*\text{-}\int_Agd\mu$. Denote by $G^1(\mu,E^*)$ (abbreviated to $G^1_{E^*}$) the space of equivalence classes of $E^*$-valued Gelfand integrable functions on $T$ with respect to weak$^*\!$ scalar equivalence, normed by
$$
\| f \|_{\mathit{G}^1}=\sup_{x\in B}\int |\langle x,f(t) \rangle| d\mu.
$$
This norm is called the \textit{Gelfand norm} and the normed space $(G^1(\mu,E^*), \|\cdot \|_{\mathit{G}^1})$, in general, is not complete.

Denote by $L^\infty(\mu)\otimes E$ the tensor product of $L^\infty(\mu)$ and $E$. A typical tensor $f^*$ in $L^\infty(\mu)\otimes E$ has a (not necessarily unique) representation $f^*=\sum_{i=1}^n\varphi_i\otimes x_i$ with $\varphi_i\in L^\infty(\mu)$, $x_i\in E$, $i=1,\dots,n$. A bilinear form on $G^1(\mu,E^*)\times (L^\infty(\mu)\otimes E)$ is given by
$$
\langle f,f^* \rangle=\sum_{i=1}^n\int \varphi_i(t)\langle x_i,f(t) \rangle d\mu=\sum_{i=1}^n\left\langle x_i,\mathit{w}^*\text{-}\int \varphi_ifd\mu \right\rangle
$$
for $f\in G^1(\mu,E^*)$ and $f^*=\sum_{i=1}^n\varphi_i\otimes x_i\in L^\infty(\mu)\otimes E$. The pair of these spaces $\langle G^1(\mu,E^*),L^\infty(\mu)\otimes E \rangle$ equipped with this bilinear form is a dual system. Thus, it is possible to define the coarsest topology on $G^1(\mu,E^*)$ such that the linear functional $f\mapsto \langle f,f^* \rangle$ is continuous for every $f^*\in L^\infty(\mu)\otimes E$, denoted by $\sigma(G^1_{E^*},L^\infty\otimes E)$, which is the topology of pointwise convergence on $L^\infty(\mu)\otimes E$. It is evident that the $\sigma(G^1_{E^*},L^\infty\otimes E)$-\hspace{0pt}topology is coarser than the weak topology $\sigma(G^1_{E^*},(G^1_{E^*})^*)$. A net $\{ f_\alpha \}$ in $G^1(\mu,E^*)$ converges to $f\in G^1(\mu,E^*)$ for the $\sigma(G^1_{E^*},L^\infty\otimes E)$-\hspace{0pt}topology if and only if  for every $x\in E$ the net $\{ \langle x,f_\alpha(\cdot) \rangle \}$ in $L^1(\mu)$ converges weakly to $\langle x,f(\cdot) \rangle \in L^1(\mu)$. 

Let $\Gamma:T\twoheadrightarrow E^*$ be a multifunction. Denote by $\overline{\mathrm{co}}^{\mathit{\,w}^*}\Gamma:T\twoheadrightarrow E^*$ the multifunction defined by the weakly$^*\!$ closed convex hull of $\Gamma(t)$. A multifunction $\Gamma$ is \textit{measurable} if the set $\{t\in T\mid \Gamma(t)\cap U\ne \emptyset \}$ is in $\Sigma$ for every weakly$^*\!$ open subset $U$ of $E^*$. If $E$ is separable, then $E^*$ is a Suslin space, and hence, a multifunction $\Gamma:T\twoheadrightarrow E^*$ with measurable graph in $\Sigma\otimes \mathrm{Borel}(E^*,\mathit{w}^*)$ admits a $\mathrm{Borel}(E^*,\mathit{w}^*)$-\hspace{0pt}measurable (or equivalently, weakly$^*\!$ measurable) selector (see \cite[Theorem III.22]{cv77}). Let $s:(\cdot, C):E\to \R\cup \{+\infty \}$ be the \textit{support function} of a set $C\subset E^*$ defined by $s(x, C)=\sup_{x^*\in C}\langle x,x^* \rangle$. A multifunction $\Gamma$ is \textit{weakly$^*\!$ scalarly measurable} if the scalar function $s(x,\Gamma(\cdot)):T\to \R\cup\{ +\infty \}$ is measurable for every $x\in E$. If $E$ is separable and $\Gamma$ has weakly$^*\!$ compact, convex values, then $\Gamma$ is scalarly measurable if and only if it is measurable (see \cite[Theorem 18.31]{ab06}). In view of $s(x,\Gamma)=s(x,\overline{\mathrm{co}}^{\mathit{\,w}^*}\Gamma)$ for every $x\in E$, if $E$ is separable and $\Gamma$ has weakly$^*$ compact values, then $\overline{\mathrm{co}}^{\mathit{\,w}^*}\Gamma$ is measurable.

Let $B^*$ be the open unit ball of $E^*$. A multifunction $\Gamma:T\twoheadrightarrow E^*$ is \textit{integrably bounded} if there exists $\varphi\in L^1(\mu)$ such that $\Gamma(t)\subset \varphi(t)B^*$ for every  $t\in T$. If $\Gamma$ is integrably bounded with measurable graph, then it admits a Gelfand integrable selector whenever $E$ is separable. Denote by $\mathcal{S}^{1,\mathit{w}^*}_\Gamma$ the set of Gelfand integrable selections of $\Gamma$. The Gelfand integral of $\Gamma$ is conventionally defined as $\mathit{w}^*\text{-}\int\Gamma d\mu:=\{ \mathit{w}^*\text{-}\int fd\mu \mid f\in \mathcal{S}^{1,\mathit{w}^*}_\Gamma \}$. If $\Gamma$ is an integrably bounded, weakly$^*$ closed, convex-valued multifunction with measurable graph, then $\mathcal{S}^{1,\mathit{w}^*}_\Gamma$ is compact in the $\sigma(G^1_{E^*},L^\infty\otimes E)$-\hspace{0pt}topology whenever $E$ is separable (see \cite[Lemma 2.1]{sa16}). 

A subset $K$ of $G^1(\mu,E^*)$ is said to be \textit{well-\hspace{0pt}dominated} if there is an integrably bounded, weakly$^*$ compact-\hspace{0pt}valued multifunction $\Gamma:T\twoheadrightarrow E^*$ such that $f(t)\in \Gamma(t)$ a.e.\ $t\in T$ for every  $f\in K$. It is evident that $K$ is well-dominated if and only if it is integrably bounded. Indeed, the dominating multifunction $\Gamma$ for $K$ is taken so as to satisfy $K\subset \varphi(t) B^*\equiv \Gamma(t)$ with $\varphi\in L^1(\mu)$. A subset $\K$ of $\M^1(\mu,E^*)$ is said to be \textit{well-dominated} if there exists an integrably bounded, weakly$^*\!$ compact-valued multifunction $\tilde{\Gamma}:T\twoheadrightarrow E^*$ such that $\Gamma(t)\subset \tilde{\Gamma}(t)$ for every $\Gamma\in \M^1(\mu,E^*)$ and $t\in T$. 

\begin{thm}[\citet{cm02}]
\label{cm1}
Let $(T,\Sigma,\mu)$ be a finite measure space and $E$ be a separable Banach space. Then a uniformly integrable subset of $G^1(\mu,E^*)$ is sequentially compact in the $\sigma(G^1_{E^*},L^\infty\otimes E)$-\hspace{0pt}topology. 
\end{thm}

\subsection{Saturation and the Lyapunov Convexity Theorem}
A finite measure space $(T,\Sigma,\mu)$ is said to be \textit{essentially countably generated} if its $\sigma$-\hspace{0pt}algebra can be generated by a countable number of subsets together with the null sets; $(T,\Sigma,\mu)$ is said to be \textit{essentially uncountably generated} whenever it is not essentially countably generated. Let $\Sigma_S=\{ A\cap S\mid A\in \Sigma \}$ be the $\sigma$-\hspace{0pt}algebra restricted to $S\in \Sigma$. Denote by $L^1_S(\mu)$ the space of $\mu$-integrable functions on the measurable space $(S,\Sigma_S)$ whose element is identified with a restriction of a function in $L^1(\mu)$ to $S$. An equivalence relation $\sim$ on $\Sigma_S$ is given by $A\sim B \Leftrightarrow \mu(A\triangle B)=0$, where $A\triangle B$ is the symmetric difference of $A$ and $B$ in $\Sigma$. The collection of equivalence classes is denoted by $\Sigma(\mu)=\Sigma/\sim$ and its generic element $\widehat{A}$ is the equivalence class of $A\in \Sigma$. We define the metric $\rho$ on $\Sigma(\mu)$ by $\rho(\widehat{A},\widehat{B})=\mu(A\triangle B)$. Then $(\Sigma(\mu),\rho)$ is a complete metric space (see \cite[Lemma 13.13]{ab06} or \cite[Lemma III.7.1]{ds58}) and $(\Sigma(\mu),\rho)$ is separable if and only if $L^1(\mu)$ is separable (see \cite[Lemma 13.14]{ab06}). The \textit{density} of $(\Sigma(\mu),\rho)$ is the smallest cardinal number of the form $|\U|$, where $\U$ is a dense subset of $\Sigma(\mu)$. 

\begin{dfn}
A finite measure space $(T,\Sigma,\mu)$ is \textit{saturated} if $L^1_S(\mu)$ is nonseparable for every $S\in \Sigma$ with $\mu(S)>0$. We say that a finite measure space has the \textit{saturation property} if it is saturated.
\end{dfn}

Saturation implies nonatomicity and several equivalent definitions for saturation are known; see \cite{fk02,fr12,hk84,ks09}. One of the simple characterizations of the saturation property is as follows. A finite measure space $(T,\Sigma,\mu)$ is saturated if and only if $(S,\Sigma_S,\mu)$ is essentially uncountably generated for every $S\in \Sigma$ with $\mu(S)>0$. The saturation of finite measure spaces is also synonymous with the uncountability of the density of $\Sigma_S(\mu)$ for every $S\in \Sigma$ with $\mu(S)>0$; see \cite[331Y(e)]{fr12}. An germinal notion of saturation already appeared in \cite{ka44,ma42}. The significance of the saturation property lies in the fact that it is necessary and sufficient for the weak/weak$^*$ compactness and the convexity of the Bochner/Gelfand integral of a multifunction as well as the Lyapunov convexity theorem in separable Banach spaces/their dual spaces. 

In particular, the following characterization of the saturation property will be useful for many applications as well as in proving the exact version of Fatou's lemma.   

\begin{prop}[\citet{ks13,ks15,po08,sy08}]
\label{lyp}
Let $(T,\Sigma,\mu)$ be a finite measure space and let $E$ be an infinite-\hspace{0pt}dimensional separable Banach space. Then the following conditions are equivalent.
\begin{enumerate}[\rm(i)]
\item $(T,\Sigma,\mu)$ is saturated. 
\item For every $\mu$-\hspace{0pt}continuous vector measure $m:\Sigma\to E$, its range $m(\Sigma)$ is weakly compact and convex in $E$.
\item For every $\mu$-\hspace{0pt}continuous vector measure $m:\Sigma\to E^*$, its range $m(\Sigma)$ is weakly$^*\!$ compact and convex in $E^*$.
\item For every integrably bounded, weakly compact-\hspace{0pt}valued multifunction $\Gamma:T\twoheadrightarrow E$ with the measurable graph, $\int\Gamma d\mu=\int\overline{\mathrm{co}}\,\Gamma d\mu$.
\item For every integrably bounded, weakly$^*\!$ compact-\hspace{0pt}valued multifunction $\Gamma:T\twoheadrightarrow E^*$ with the measurable graph, $\mathit{w}^*\text{-}\int\Gamma d\mu=\mathit{w}^*\text{-}\int\overline{\mathrm{co}}^{\,\mathit{w}^*}\Gamma d\mu$.
\end{enumerate}
In particular, the implications (i) $\Rightarrow$ (ii), (iii), (iv), (v) are true for every separable Banach space. 
\end{prop}

\begin{rem}
The equivalence (i) $\Leftrightarrow$ (ii) $\Leftrightarrow$ (iii) is proven by \cite{ks13,ks15} and (i) $\Leftrightarrow$ (iv) $\Leftrightarrow$ (v) is due to \cite{po08,sy08}. The equivalence of saturation and the ``bang-\hspace{0pt}bang principle'' in separable Banach spaces/their dual spaces is demonstrated by \cite{ks14b,ks16a,sa16}. For the equivalence of saturation with the ``purification principle'', see \cite{ks14b,ls09,po09,sa16}, for that with the convexity of the distribution of a multifunction, see \cite{ks09,ks14b}, and for that with the ``minimization principle'', see \cite{sa16}. As an application to game theory, \cite{ks09} provides an intriguing characterization of the saturation property in terms of the existence of Nash equilibria in large games.\fn{We emphasize yet again that in our  ignoring of large parts of the applied mathematics literature, we ignore the role of Fatou's lemma in the convergence of set-valued conditional expectations and the strong law of large numbers for multifunctions; see  \cite{hi85,hu77,lok02,mo05} for example.} 
\end{rem}

\section{Fatou's Lemma Based on Bochner Integration}
\label{fat2}
We established in Subsection \ref{fat1} a prototype of Fatou's lemma in finite dimensions (Theorem \ref{thm1}). The purpose of this section is to show how it can be  extended to a separable Banach space setting under Bochner integration. We first show an approximate version of Fatou's lemma under nonatomicity and discuss the hurdles that need to be overcome  in removing  the approximate operation.  To overcome this difficulty, we show  the saturation of measure spaces is inevitable to establish an exact version of Fatou's lemma. Finally, we characterize the saturation property itself in terms of the exact Fatou's lemma.

\subsection{Approximate Fatou's Lemma}
\label{app1}
The \textit{strong upper limit} of a sequence $\{ F_n \}_{n\in \N}$ of subsets in $E$ is defined by
$$
\mathrm{Ls}\left\{ F_n \right\}=\left\{ x\in E \mid \exists \left\{ x_{n_i} \right\}_{i\in \N}: x=\lim_{i\to \infty}x_{n_i},\,x_{n_i}\in F_{n_i}\,\forall i\in \mathbb{N} \right\},
$$
where $\{ x_{n_i} \}_{i\in \N}$ denotes a subsequence of $\{ x_n \}_{n\in \N}\subset E$. We denote by $\mathit{w}\text{-}\lim_{n}x_n$ the weak limit point of a sequence $\{ x_n \}_{n\in \N}$ in $E$. The \textit{weak upper limit} of $\{ F_n \}_{n\in \N}$ of subsets in $E$ is defined by
$$
\mathit{w}\text{-}\mathrm{Ls}\left\{ F_n \right\}=\left\{ x\in E \mid \exists \left\{ x_{n_i} \right\}_{i\in \N}: x=\mathit{w}\text{-}\lim_{i\to \infty}x_{n_i},\,x_{n_i}\in F_{n_i}\,\forall i\in \mathbb{N} \right\}.
$$
The notion of upper limits plays a central role in the  formulation of  Fatou's lemma discussed  in this essay; for  a detailed treatment on the  calculus of limits of sets, see \cite{af90,hi74,mo06}. 

Let $\{ f_n \}_{n\in \N}$ be a sequence in $L^1(\mu,E)$. Since the essential range of a Bochner integrable function is separable by the Pettis measurability theorem (see \cite[Theorem II.1.2]{du77}), we may assume without loss of generality that the range of each $f_n$ is contained in a common separable Banach space. To this end, let $V_n$ be the essential range of $f_n$ and take the linear span of $V:=\bigcup_{n\in \N}V_n$. Then each $f_n$ essentially takes values in the separable subspace space $V$ of $E$. This observation guarantees that the multifunction $t\mapsto \mathit{w}\text{-}\mathrm{Ls}\,\{ f_n(t) \}$ from $T$ into $E$ is measurable whenever $\{ f_n \}_{n\in \N}$ is well-dominated  (see \cite[Proposition 4.3]{he90}). Furthermore, the multifunction $t\mapsto \overline{\mathrm{co}}\,\mathit{w}\text{-}\mathrm{Ls}\,\{ f_n(t) \}$ is measurable, too (see \cite[Theorem 8.2.2]{af90}). 

Our reference point under investigation is the following result due to \cite{km86}, which is the first work in the literature on Fatou's lemma in infinite-\hspace{0pt}dimensions, where the slightly improved version of the original result with the current form was given in \cite{ks14a}.\footnote{For another variant of an approximate Fatou's lemma in Banach spaces, see \cite{ba88}.} 

\begin{thm}[\citet{km86}]
\label{km1}
Let $(T,\Sigma,\mu)$ be a finite measure space and $E$ be a Banach space. If $\{ f_n \}_{n\in \N}$ is a well-\hspace{0pt}dominated sequence in $L^1(\mu,E)$, then for every $\varepsilon>0$ there exists $f\in L^1(\mu,E)$ with the following properties. 
\begin{enumerate}[\rm (i)]
\item $f(t)\in \mathit{w}\text{-}\mathrm{Ls}\left\{ f_n(t) \right\}$ a.e.\ $t\in T$;
\item $\displaystyle\int fd\mu\in \mathit{w}\text{-}\mathrm{Ls}\left\{ \int f_nd\mu \right\}+\varepsilon B$.
\end{enumerate}
\end{thm}

To prove Theorem \ref{km1}, an infinite-dimensional analogue of Lemma \ref{lem1} is required (see \cite{km86,ya88,ya89,ya91} for details), which also plays a crucial role in proving  the main result of the essay.

\begin{lem}[\citet{km86}]
\label{lem2}
Let $(T,\Sigma,\mu)$ be a finite measure space and $E$ be a Banach space. If $\{ f_n \}_{n\in \N}$ is a well-\hspace{0pt}dominated sequence such that $f_n\to f$ weakly in $L^1(\mu,E)$, then
$$
f(t)\in \overline{\mathrm{co}}\,\mathit{w}\text{-}\mathrm{Ls}\left\{ f_n(t) \right \}\quad \text{a.e.\ $t\in T$}.
$$
\end{lem}

Under the nonatomicity assumption, however, the $\varepsilon$-\hspace{0pt}approximation (or the closure operation) cannot be removed from Theorem \ref{km1}. Such a counterexample was constructed by \cite{ru89} based on the famous failure of the Lyapunov convexity theorem for an $l^2$-valued vector measure on the Lebesgue unit interval and then fortified with any infinite-dimensional Banach space by \cite{ks14a} to the current general form. 

\begin{prop}[\citet{ks14a,ru89}]
\label{rust1}
For every essentially countably generated, nonatomic, finite measure space $(T,\Sigma,\mu)$ and infinite-\hspace{0pt}dimensional Banach space $E$, there exists a well-\hspace{0pt}dominated sequence $\{ f_n \}_{n\in \N}$ in $L^1(\mu,E)$ with the following properties. 
\begin{enumerate}[\rm (i)]
\item $\displaystyle\mathrm{Ls}\left\{ \int f_nd\mu \right\}\not\subset \int \mathit{w}\text{-}\mathrm{Ls}\left\{ f_n  \right\}d\mu$. 
\item There exists no $f\in L^1(\mu,E)$ such that
\begin{enumerate}[\rm (a)]
\item $f(t)\in \mathit{w}\text{-}\mathrm{Ls}\left\{ f_n(t) \right\}$ a.e.\ $t\in T$; 
\item $\displaystyle\int fd\mu\in \mathit{w}\text{-}\mathrm{Ls}\left\{ \int f_nd\mu \right\}$.
\end{enumerate}
\end{enumerate}
\end{prop}

Since the inclusion $\mathrm{Ls}\,\{ \int f_nd\mu \}\subset \mathit{w}\text{-}\mathrm{Ls}\,\{ \int f_nd\mu \}$ is automatic, this counterexample is stronger than the negation: $\mathit{w}\text{-}\mathrm{Ls}\,\{ \int f_nd\mu \}\not\subset \int \mathit{w}\text{-}\mathrm{Ls}\,\{ f_n \}d\mu$. The failure of an ``exact'' Fatou lemma in infinite dimensions leads, as an inevitable consequence, to the introduction of the saturation property on measure spaces.

\subsection{Exact Fatou's Lemma}
To illustrate the power of saturation, we provide the proof of an exact version of Fatou's lemma for completeness.   

\begin{thm}[\citet{ks14a}]
\label{ks}
Let $(T,\Sigma,\mu)$ be a saturated finite measure space and $E$ be a Banach space. If $\{ f_n \}_{n\in \N}$ is a well-\hspace{0pt}dominated sequence in $L^1(\mu,E)$, then
\begin{enumerate}[\rm (i)]
\item $\mathit{w}\text{-}\mathrm{Ls}\left\{ \displaystyle\int f_nd\mu \right\}\subset \displaystyle\int\mathit{w}\text{-}\mathrm{Ls}\left\{ f_n \right\}d\mu$. 
\item There exists $f\in L^1(\mu,E)$ such that
\begin{enumerate}[\rm (a)]
\item $f(t)\in \mathit{w}\text{-}\mathrm{Ls}\left\{ f_n(t) \right\}$ a.e.\ $t\in T$;
\item $\displaystyle\int fd\mu\in \mathit{w}\text{-}\mathrm{Ls}\left\{ \int f_nd\mu \right\}$.
\end{enumerate}
\end{enumerate}
\end{thm}

\begin{proof}
Since the well-dominated sequence $\{ f_n \}_{n\in \N}$ is relatively weakly compact in $L^1(\mu,E)$, by Theorem \ref{diest}, one can extract from it a subsequence $\{ f_{n_i} \}_{i\in \N}$ that converges weakly to $f_0$ in $L^1(\mu,E)$. Since the integration operator $g\mapsto \int gd\mu$ from $L^1(\mu,E)$ to $E$ is continuous in the weak topologies for $L^1(\mu,E)$ and $E$, we obtain $\int f_id\mu\to \int f_0d\mu$ weakly in $E$, and hence, $\int f_0d\mu\in \mathit{w}\text{-}\mathrm{Ls}\,\{ \int f_nd\mu \}$. It follows from Lemma \ref{lem2} that $f_0(t)\in \overline{\mathrm{co}}\,\mathit{w}\text{-}\mathrm{Ls}\,\{ f_n(t) \}$ a.e.\ $t\in T$. Integrating this inclusion yields 
$$
\int f_0d\mu\in \int\overline{\mathrm{co}}\,\mathit{w}\text{-}\mathrm{Ls}\left\{ f_n \right\}d\mu=\int\mathit{w}\text{-}\mathrm{Ls}\left\{ f_n \right\}d\mu,
$$
where the equality follows from Proposition \ref{lyp}. Hence, there exists a Bochner integrable selector $f$ of $ \mathit{w}\text{-}\mathrm{Ls}\,\{ f_n \}$ such that $\int fd\mu=\int f_0d\mu\in \mathit{w}\text{-}\mathrm{Ls}\,\{ \int f_nd\mu \}$, which verifies condition (ii). 

Condition (i) follows easily from condition (ii). To show this claim, take any $x\in \mathit{w}\text{-}\mathrm{Ls}\,\{ \int f_nd\mu \}$. Then there exists a subsequence $\{ f_{n_i} \}_{i\in \N}$ such that $\mathit{w}\text{-}\lim_i\int f_{n_i}d\mu=x$. It follows from condition (ii) that there exists $f\in L^1(\mu,E)$ such that (a) $f(t)\in \mathit{w}\text{-}\mathrm{Ls}\left\{ f_{n_i}(t) \right\}$ a.e.\ $t\in T$; (b) $\int fd\mu=\mathit{w}\text{-}\lim_j\int f_{n_j}d\mu$, where $\{ f_{n_j} \}_{j\in \N}$ is a further subsequence of $\{ f_{n_i} \}_{i\in \N}$. Integrating the inclusion (a) together with condition (b) yields $x\in \int\mathit{w}\text{-}\mathrm{Ls}\,\{ f_n \}d\mu$.
\end{proof}

\begin{cor}[\citet{ks14a}]
\label{cor1}
Let $(T,\Sigma,\mu)$ be a saturated finite measure space and $E$ be a separable Banach space. If $\{ \Gamma_n \}_{n\in \N}$ is a well-dominated sequence of multifunctions in $\M^1(\mu,E)$, then:
$$
\mathit{w}\text{-}\mathrm{Ls}\left\{ \int\Gamma_n d\mu \right\}\subset \int \mathit{w}\text{-}\mathrm{Ls}\,\{ \Gamma_n \}d\mu.
$$
\end{cor}

\begin{proof}
If $\mathit{w}\text{-}\mathrm{Ls}\,\{ \int\Gamma_n d\mu \}=\emptyset$, then the result is trivially true. Thus, without loss of generality, we may assume that $\mathit{w}\text{-}\mathrm{Ls}\,\{ \int\Gamma_n d\mu \}\ne\emptyset$. Take any $x\in \mathit{w}\text{-}\mathrm{Ls}\,\{ \int \Gamma_nd\mu \}$. Then there is a sequence $\{ x_n \}_{n\in \N}$ in $E$ with $x_n\in \int \Gamma_nd\mu$ for each $n$ such that $x_{n_i}\to x$ weakly. Hence, there is an integrably bounded sequence $\{ f_{n_i} \}_{i\in \N}$ in $L^1(\mu,E)$ such that $x_{n_i}=\int f_{n_i}d\mu$ and $f_{n_i}\in \mathcal{S}^1_{\Gamma_{n_i}}$ for each $i$. It follows from Theorem \ref{ks} that
$$
x=\mathit{w}\text{-}\lim_{i\to \infty}x_{n_i}\in \mathit{w}\text{-}\mathrm{Ls}\left\{ \int f_nd\mu \right\}\subset \int \mathit{w}\text{-}\mathrm{Ls}\left\{ f_n \right\}d\mu\subset \int \mathit{w}\text{-}\mathrm{Ls}\left\{ \Gamma_n \right\}d\mu.
$$
Therefore, the desired inclusion holds.
\end{proof}

Following the terminology of the finite-dimensional case, given a sequence $\{ f_n \}_{n\in \N}$ in $L^1(\mu,E)$ such that $\mathit{w}\text{-}\mathrm{Ls}\,\{ \int f_nd\mu \}$ is nonempty, condition (i) of Theorem \ref{ks} is referred to the \textit{weak Fatou property} and condition (ii) of Theorem \ref{ks} is referred to the \textit{weak upper closure property}. Similar to the function case, given a sequence of multifunctions $\{ \Gamma_n \}_{n\in \N}$ in $\M^1(\mu,E)$ such that $\mathit{w}\text{-}\mathrm{Ls}\,\{ \int \Gamma_nd\mu \}$ is nonempty, the inclusion of Corollary \ref{cor1} is referred to the \textit{weak Fatou property}. 

We now turn to the saturation property of nonatomic finite measure spaces, and show that it can be completely characterized by the weak Fatou property and the weak closure property and formalized as follows. 

\begin{prop}[\citet{ks14a}]
\label{eqv1}
Let $(T,\Sigma,\mu)$ be a nonatomic finite measure space and $E$ be an infinite-\hspace{0pt}dimensional Banach space. Then the following conditions are equivalent.
\begin{enumerate}[\rm (i)]
\item $(T,\Sigma,\mu)$ has the saturation property.
\item Every well-\hspace{0pt}dominated sequence of functions in $L^1(\mu,E)$ has the weak Fatou property.
\item Every well-\hspace{0pt}dominated sequence of functions in $L^1(\mu,E)$ has the weak upper closure property.
\item Every well-\hspace{0pt}dominated sequence of multifunctions functions in $\M^1(\mu,E)$ has the weak Fatou property.
\end{enumerate}
\end{prop}

\begin{proof}
The implication (i) $\Rightarrow$ (iii) $\Rightarrow$ (ii) $\Rightarrow$ (iv) is already shown in the proof of Theorem \ref{ks} and Corollary \ref{cor1}. We show the implication (iv) $\Rightarrow$ (i). Suppose that a nonatomic finite measure space $(T,\Sigma,\mu)$ is not saturated. Then there exists $S\in \Sigma$ with $\mu(S)>0$ such that $(S,\Sigma_S,\mu)$ is countably generated. Since $\mu$ is nonatomic, it follows from Proposition \ref{rust1} that there exists a well-\hspace{0pt}dominated sequence of Bochner integrable functions $\{ f_n \}_{n\in \N}$ from $S$ to $E$ such that the weak Fatou property fails to hold for $(S,\Sigma_S,\mu)$. Let $\Gamma:S\twoheadrightarrow E$ be a dominating multifunction for $\{ f_n \}_{n\in \N}$. Extend the functions to $T$ by $\tilde{f}_n(t)=f_n(t)$ if $t\in S$ and $\tilde{f}_n(t)=\{ 0 \}$ otherwise, and similarly $\tilde{\Gamma}(t)=\Gamma(t)$ if $t\in S$ and $\tilde{\Gamma}(t)=\{ 0 \}$ otherwise. By construction, the well-\hspace{0pt}dominated sequence of Bochner integrable functions $\{ \tilde{f}_n \}_{n\in \N}$ fails to satisfy the weak Fatou property for $(T,\Sigma,\mu)$.
\end{proof}

\section{Fatou's Lemma Based on Gelfand Integration}
We demonstrate in this section how the prototype Fatou's lemma in finite dimensions is extended to the dual space of a separable Banach space with Gelfand integrals. As in Section \ref{fat2}, we first show an approximate version of Fatou's lemma under nonatomicity and then derive an exact version of Fatou's lemma under the saturation hypothesis. Finally, we derive the necessity of saturation for Fatou's lemma. Although the proofs need an independent treatment, their method is parallel to the one developed in Section \ref{fat2} with a suitable replacement of the weak topology and Bochner integrals by the weak$^* \!$ topology and Gelfand integrals respectively.

\subsection{Approximate Fatou's Lemma}
The \textit{strong upper limit} of a sequence $\{ F_n \}_{n\in \N}$ of subsets in $E^*$ is defined by
$$
\mathrm{Ls}\left\{ F_n \right\}=\left\{ x^*\in E^* \mid \exists \left\{ x^*_{n_i} \right\}_{i\in \N}: x^*=\lim_{i\to \infty}x^*_{n_i},\,x^*_{n_i}\in F_{n_i}\,\forall i\in \mathbb{N} \right\},
$$
where $\{ x^*_{n_i} \}_{i\in \N}$ denotes a subsequence of $\{ x^*_n \}_{n\in \N}\subset E^*$ and the convergence is with respect to the dual norm in $E^*$. We denote by $\mathit{w}^*\text{-}\lim_{n}x_n$ the weak$^*\!$ limit point of a sequence $\{ x^*_n \}$ in $E^*$. The \textit{weak$^*\!$ upper limit} of $\{ F_n \}_{n\in \N}$ is defined by
$$
\mathit{w}^*\text{-}\mathrm{Ls}\left\{ F_n \right\}=\left\{ x^*\in E^* \mid \exists \{ x^*_{n_i} \}_{i\in \N}: x^*=\mathit{w}^*\text{-}\lim_{i\to \infty}x_{n_i}^*,\,x^*_{n_i}\in F_{n_i}\,\forall i\in \mathbb{N} \right\}.
$$

If $\{ f_n \}_{n\in \N}$ is an integrably bounded sequence in $G^1(\mu,E^*)$, then the multifunction from $T$ into $E^*$ defined by $t\mapsto \mathit{w}^*\text{-}\mathrm{Ls}\,\{ f_n(t) \}$ is measurable whenever $E$ is separable (see \cite[Lemma 4.5]{bs05}). Furthermore, the multifunction $t\mapsto\overline{\mathrm{co}}^{\,\mathit{w}^*}\mathit{w}^*\text{-}\mathrm{Ls}\,\{ f_n(t) \}$ is measurable as noted in Subsection \ref{gel}. 

A Gelfand integral analogue of Theorem \ref{km1} is the following result.\footnote{As shown in \cite{ba02,cs09}, the integrable boundedness condition in Theorem \ref{cm2} and Lemma \ref{lem3} can be extended to the uniform integrability condition.}

\begin{thm}[\citet{cm02}]
\label{cm2}
Let $(T,\Sigma,\mu)$ be a complete finite measure space and $E$ be a separable Banach space. If $\{ f_n \}_{n\in \N}$ is an integrably bounded sequence in $G^1(\mu,E^*)$, then for every $\varepsilon>0$ there exists $f\in G^1(\mu,E^*)$ with the following properties. 
\begin{enumerate}[\rm (i)]
\item $f(t)\in \mathit{w}^*\text{-}\mathrm{Ls}\left\{ f_n(t) \right\}$ a.e.\ $t\in T$;
\item $\displaystyle\mathit{w}^*\text{-}\int fd\mu \in \mathit{w}^*\text{-}\mathrm{Ls}\left\{ \mathit{w}^*\text{-}\int f_n d\mu \right\}+\varepsilon B^*$.
\end{enumerate}
\end{thm}

Similarly, a Gelfand integral analogue of Lemma \ref{lem2} is the following result.

\begin{lem}[\citet{cm02}]
\label{lem3}
\label{km}Let $(T,\Sigma,\mu)$ be a finite measure space and $E$ be a separable Banach space. If $\{ f_n \}_{n\in \N}$ is an integrably bounded sequence in $G^1(\mu,E^*)$ such that $f_n\to f$ in the $\sigma(G^1_{E^*},L^\infty\otimes E)$-\hspace{0pt}topology, then: 
$$
f(t)\in \overline{\mathrm{co}}^{\,\mathit{w}^*}\mathit{w}^*\text{-}\mathrm{Ls}\left\{ f_n(t) \right \}\quad \text{a.e.\ $t\in T$}.
$$
\end{lem}

Similar to Proposition \ref{rust1}, nonatomicity is insufficient to remove the $\varepsilon$-\hspace{0pt}approximation (or the closure operation) from Theorem \ref{cm2}. Specifically, the following negative result holds. 

\begin{prop}[\citet{kss16,ru89}]
\label{rust2}
For every nonatomic finite measure space $(T,\Sigma,\mu)$ that is not saturated and for every infinite-\hspace{0pt}dimensional Banach space $E$, there exists an integrably bounded sequence $\{ f_n \}_{n\in \N}$ in $G^1(\mu,E^*)$ with the following properties.
\begin{enumerate}[\rm (i)]
\item $\displaystyle\mathrm{Ls}\left\{ \mathit{w}^*\text{-}\int f_nd\mu \right\}\not\subset \mathit{w}^*\text{-}\int \mathit{w}^*\text{-}\mathrm{Ls}\left\{ f_n  \right\}d\mu$.
\item There exists no $f\in G^1(\mu,E^*)$ such that
\begin{enumerate}[\rm (a)]
\item $f(t)\in \mathit{w}^*\text{-}\mathrm{Ls}\left\{ f_n(t) \right\}$ a.e.\ $t\in T$;
\item $\displaystyle\mathit{w}^*\text{-}\int fd\mu\in \mathit{w}^*\text{-}\mathrm{Ls}\left\{ \mathit{w}^*\text{-}\int f_nd\mu \right\}$.
\end{enumerate}
\end{enumerate}
\end{prop}

The similar remark with Proposition \ref{rust1} is valid here: Since the inclusion $\mathrm{Ls}\,\{ \int f_nd\mu \}\subset \mathit{w}^*\text{-}\mathrm{Ls}\,\{ \mathit{w}^*\text{-}\int f_nd\mu \}$ is automatic, this counterexample is stronger than the negation: $\mathit{w}^*\text{-}\mathrm{Ls}\,\{ \mathit{w}^*\text{-}\int f_nd\mu \}\not\subset \mathit{w}^*\text{-}\int \mathit{w}^*\text{-}\mathrm{Ls}\,\{ f_n \}d\mu$.

\subsection{Exact Fatou's Lemma}
\begin{thm}[\citet{kss16}]
\label{kss}
Let $(T,\Sigma,\mu)$ be a saturated finite measure space and $E$ be a separable Banach space. If $\{ f_n \}_{n\in \N}$ is an integrably bounded sequence in $G^1(\mu,E^*)$, then:
\begin{enumerate}[\rm (i)]
\item $\mathit{w}^*\text{-}\mathrm{Ls}\left\{ \displaystyle\mathit{w}^*\text{-}\int f_nd\mu \right\}\subset \displaystyle\mathit{w}^*\text{-}\int\mathit{w}^*\text{-}\mathrm{Ls}\left\{ f_n \right\}d\mu$.
\item There exists $f\in G^1(\mu,E^*)$ such that
\begin{enumerate}[\rm(a)]
\item $f(t)\in \mathit{w}^*\text{-}\mathrm{Ls}\left\{ f_n(t) \right\}$ a.e.\ $t\in T$;
\item $\displaystyle\mathit{w}^*\text{-}\int fd\mu \in \mathit{w}^*\text{-}\mathrm{Ls}\left\{ \mathit{w}^*\text{-}\int f_n d\mu \right\}$.
\end{enumerate}
\end{enumerate}
\end{thm}

\begin{proof}
Since the sequence $\{ f_n \}_{n\in \N}$ has a subsequence that converges to some $f_0\in G^1(\mu,E^*)$ in the $\sigma(G^1_{E^*},L^\infty\otimes E)$-\hspace{0pt}topology by Theorem \ref{cm1}, we have
\begin{align*}
\left\langle x,\lim_{i\to \infty}\mathit{w}^*\text{-}\int f_{n_i}d\mu \right\rangle=\lim_{i\to \infty}\int\langle x,f_{n_i}(t)\rangle d\mu
& =\int\langle x,f_0(t)\rangle d\mu \\
& =\left\langle x,\mathit{w}^*\text{-}\int f_0d\mu \right\rangle
\end{align*}
for every $x\in E$. Hence, $\mathit{w}^*\text{-}\int f_0d\mu\in \mathit{w}^*\text{-}\mathrm{Ls}\left\{ \mathit{w}^*\text{-}\int f_nd\mu \right\}$. By Lemma \ref{lem3}, we have $f_0(t)\in \overline{\mathrm{co}}^{\,\mathit{w}^*}\mathit{w}^*\text{-}\mathrm{Ls}\left\{ f_n(t) \right\}$ a.e.\ $t\in T$. Integrating this inclusion yields
$$
\mathit{w}^*\text{-}\int f_0d\mu\in \mathit{w}^*\text{-}\int\overline{\mathrm{co}}^{\,\mathit{w}^*}\mathit{w}^*\text{-}\mathrm{Ls}\left\{ f_n \right\}d\mu=\mathit{w}^*\text{-}\int\mathit{w}^*\text{-}\mathrm{Ls}\left\{ f_n \right\}d\mu.
$$
Therefore, there exists a Gelfand integrable selector $f$ of $\mathit{w}^*\text{-}\mathrm{Ls}\,\{ f_n \}$ such that $\mathit{w}^*\text{-}\int fd\mu=\mathit{w}^*\text{-}\int f_0d\mu \in \mathit{w}^*\text{-}\mathrm{Ls}\,\{ \mathit{w}^*\text{-}\int f_n d\mu \}$, which verifies condition (i). 

To show condition (ii), take any $x^*\in \mathit{w}^*\text{-}\mathrm{Ls}\,\{ \mathit{w}^*\text{-}\int f_nd\mu \}$. Then there exists a subsequence $\{ f_{n_i} \}_{i\in \N}$ such that $\mathit{w}^*\text{-}\lim_i\mathit{w}^*\text{-}\int f_{n_i}d\mu=x^*$. It follows from condition (i) that there exists $f\in G^1(\mu,E^*)$ such that (a) $f(t)\in \mathit{w}^*\text{-}\mathrm{Ls}\,\{ f_{n_i}(t) \}$ a.e.\ $t\in T$; and (b) $\mathit{w}^*\text{-}\int fd\mu=\mathit{w}^*\text{-}\lim_j\mathit{w}^*\text{-}\int f_{n_j}d\mu$, where $\{ f_{n_j} \}_{j\in \N}$ is a further subsequence of $\{ f_{n_i} \}_{i\in \N}$. Integrating the inclusion (a) together with condition (b) yields $x^*=\mathit{w}^*\text{-}\int fd\mu\in \int\mathit{w}^*\text{-}\mathrm{Ls}\,\{ f_n \}d\mu$.
\end{proof}

\begin{cor}[\citet{kss16}]
\label{cor2}
Let $(T,\Sigma,\mu)$ be a saturated finite measure space and $E$ be a separable Banach space. If $\{ \Gamma_n \}_{n\in \N}$ is a well-dominated sequence of multifunctions in $\M^1(\mu,E^*)$, then:
$$
\mathit{w}^*\text{-}\mathrm{Ls}\left\{ \mathit{w}^*\text{-}\int\Gamma_n d\mu \right\}\subset \mathit{w}^*\text{-}\int \mathit{w}^*\text{-}\mathrm{Ls}\left\{ \Gamma_n \right\}d\mu.
$$
\end{cor}

\begin{proof}
If $\mathit{w}^*\text{-}\mathrm{Ls}\,\{ \mathit{w}^*\text{-}\int\Gamma_n d\mu \}=\emptyset$, then the result is trivially true. Thus, without loss of generality, we may assume that $\mathit{w}^*\text{-}\mathrm{Ls}\,\{ \mathit{w}^*\text{-}\int\Gamma_n d\mu \}\ne\emptyset$. Take any $x\in \mathit{w}^*\text{-}\mathrm{Ls}\,\{ \mathit{w}^*\text{-}\int \Gamma_nd\mu \}$. Then there is a sequence $\{ x_n \}_{n\in \N}$ in $E^*$ with $x_n\in \mathit{w}^*\text{-}\int \Gamma_nd\mu$ for each $n$ such that $x_{n_i}\to x$ weakly$^*$. Hence, there is an integrably bounded sequence $\{ f_{n_i} \}_{i\in \N}$ in $G^1(\mu,E^*)$ such that $x_{n_i}=\mathit{w}^*\text{-}\int f_{n_i}d\mu$ and $f_{n_i}\in \mathcal{S}^{1,\mathit{w}^*}_{\Gamma_{n_i}}$ for each $i$. It follows from Theorem \ref{kss} that
\begin{align*}
x=\mathit{w}^*\text{-}\lim_{i\to \infty}x_{n_i}\in \mathit{w}^*\text{-}\mathrm{Ls}\left\{ \mathit{w}^*\text{-}\int f_nd\mu \right\}
& \subset \int \mathit{w}^*\text{-}\mathrm{Ls}\left\{ f_n \right\}d\mu \\
& \subset \mathit{w}^*\text{-}\int \mathit{w}^*\text{-}\mathrm{Ls}\left\{ \Gamma_n \right\}d\mu.
\end{align*}
Therefore, the desired inclusion holds.
\end{proof}

For the earlier exact result in nonatomic Loeb measure spaces, see \cite{ls07,su97}.   

For the sake of simplicity, given a sequence $\{ f_n \}_{n\in \N}$ in $G^1(\mu,E^*)$ such that $\mathit{w}^*\text{-}\mathrm{Ls}\,\{ \mathit{w}^*\text{-}\int f_nd\mu \}$ is nonempty, condition (i) of Theorem \ref{kss} is referred to the \textit{weak$^*\!$ Fatou property} and condition (ii) of Theorem \ref{kss} is referred to the \textit{weak$^*\!$ upper closure property}. Similar to the function case, given a sequence of multifunctions $\{ \Gamma_n \}_{n\in \N}$ in $\M^1(\mu,E^*)$ such that $\mathit{w}^*\text{-}\mathrm{Ls}\,\{ \mathit{w}^*\text{-}\int \Gamma_nd\mu \}$ is nonempty, the inclusion of Corollary \ref{cor2} is referred to the \textit{weak$^*\!$ Fatou property}. 

\begin{prop}[\citet{kss16}]
Let $(T,\Sigma,\mu)$ be a nonatomic finite measure space and $E$ be an infinite-\hspace{0pt}dimensional separable Banach space. Then the following conditions are equivalent:
\begin{enumerate}[\rm (i)]
\item $(T,\Sigma,\mu)$ has the saturation property.
\item Every integrably bounded sequence of functions in $G^1(\mu,E^*)$ has the weak$^*\!$ Fatou property.
\item Every integrably bounded sequence of functions in $G^1(\mu,E^*)$ has the weak$^*\!$ upper-\hspace{0pt}closure property.
\item Every well-dominated sequence of multifunctions in $\M^1(\mu,E^*)$ has the weak$^*\!$ Fatou property.
\end{enumerate}
\end{prop}

\begin{proof}
The implication (i) $\Rightarrow$ (iii) $\Rightarrow$ (ii) $\Rightarrow$ (iv) is already shown in the proof of Theorem \ref{kss} and Corollary \ref{cor2}. We show the implication (iv) $\Rightarrow$ (i). Suppose that a nonatomic finite measure space $(T,\Sigma,\mu)$ is not saturated. Then there exists $S\in \Sigma$ with $\mu(S)>0$ such that $(S,\Sigma_S,\mu)$ is countably generated. Since $\mu$ is nonatomic, it follows from Proposition \ref{rust2} that there exists a well-\hspace{0pt}dominated sequence of Gelfand integrable functions $\{ f_n \}_{n\in \N}$ from $S$ to $E^*$ such that the weak$^*\!$ Fatou property fails to hold for $(S,\Sigma_S,\mu)$. Let $\Gamma:S\twoheadrightarrow E^*$ be a dominating multifunction for $\{ f_n \}_{n\in \N}$. Extend the functions to $T$ by $\tilde{f}_n(t)=f_n(t)$ if $t\in S$ and $\tilde{f}_n(t)=\{ 0 \}$ otherwise, and similarly $\tilde{\Gamma}(t)=\Gamma(t)$ if $t\in S$ and $\tilde{\Gamma}(t)=\{ 0 \}$ otherwise. By construction, the well-\hspace{0pt}dominated sequence of Gelfand integrable functions $\{ \tilde{f}_n \}_{n\in \N}$ fails to satisfy the weak$^*\!$ Fatou property for $(T,\Sigma,\mu)$.
\end{proof}

\section{Galerkin Approximations with Finite-\hspace{0pt}Dimensional Projections}
In this section, we describe some useful properties of projections in locally convex Hausdorff spaces (lcHs) to deal with Banach spaces with the norm topology and their dual spaces with the weak$^*$ topology simultaneously. Furthermore, we provide an intimate relation between finite-dimensional projections and Galerkin approximation schemes.

\subsection{Projections in LcHs}
Let $E$ be a locally convex Hausdorff space. For a vector subspace $V$ of $E$, we denote by $V^\perp$ the \textit{annihilator} of $V$, i.e., $V^\perp=\{ p\in E^*\mid \langle p,x \rangle=0 \ \forall x\in V \}$. If $V$ is a closed vector subspace of $E$ topologically complemented in $E$, then there is a continuous projection $P$ of $E$ onto $V$ such that $E=V\oplus \mathcal{R}(I-P)$, where $\mathcal{R}(I-P)$ is the range of the continuous projection $I-P$ and it is a topologically complemented subspace of $V$ in $E$ (see \cite[\S10, 7(6)]{ko69}). Let $P^*:E^*\to E^*$ be the adjoint operator of $P$. Then $P^*$ is a continuous projection of $E^*$ onto $V^*$ and $I-P^*$ is a continuous projection of $E^*$ onto $(\mathcal{R}(I-P))^*$, and hence, $E^*=V^*\oplus(\mathcal{R}(I-P))^*$, where $E^*=V^*\oplus(\mathcal{R}(I-P))^*$ and $V^*=\mathcal{R}(I-P)^\perp$ and $\mathcal{R}(I-P^*)=V^\perp$ (see \cite[\S10, 8(5), (6)]{ko69}). 

In particular, if $V^n$ is an $n$-dimensional vector subspace of $E$, there is a continuous projection $P_n$ of $E$ onto $V^n$ such that $E=V^n\oplus \mathcal{R}(I-P_n)$ and $\mathcal{R}(I-P_n)$ is a topologically complemented subspace of $V^n$ in $E$ (see \cite[\S10, 7(8)]{ko69}). It is easy to see that $P_nP_m=P_mP_n=P_{\min\{ m,n \}}$ for each $m,n\in \N$. Let $P_n^*:E^*\to E^*$ be the adjoint operator of $P_n$. Then $P_n^*$ is a continuous  projection of $E^*$ onto $(V^n)^*=\mathcal{R}(I-P_n)^\perp$ and $I-P_n^*$ is a projection of $E^*$ onto $(\mathcal{R}(I-P_n))^*=(V^n)^\perp$. It is easy to see that $P_n^*P_m^*=P_m^*P_n^*=P_{\min\{ m,n \}}^*$. 

The following simple observation is employed for the finite truncation method in separable Banach spaces and $L^\infty$ explored in Section \ref{sec}. 

\begin{thm}
\label{prj1}
Let $E$ be a lcHs and $P$ be a continuous projection of $E$ onto a closed vector subspace $V$. 
\begin{enumerate}[\rm(i)]
\item If $\{ x_n \}_{n\in \N}$ is a sequence in $E$ converging weakly to $x\in V$, then $\{ Px_n \}_{n\in \N}$ is a sequence in $E$ converging weakly to $x$. 
\item If $\{ p_n \}_{n\in \N}$ is a sequence in $E^*$ converging weakly$^*$ to $p\in V^*$, then $\{ P^*p_n \}_{n\in \N}$ is a sequence in $E^*$ converging weakly$^*$ to $p$. 
\end{enumerate}
\end{thm}

\begin{proof}
(i): In view of $Px=x$, for every $p\in E^*$: 
\begin{align*}
\lim_{n\to \infty}\langle p,Px_n \rangle=\lim_{n\to \infty}\langle P^*p,x_n \rangle=\langle P^*p,x \rangle=\langle p,Px \rangle=\langle p,x \rangle.
\end{align*}
Thus, $Px_n\to x$ weakly in $E$. 

(ii) Since $P^*$ is the projection of $E^*$ onto $V^*$, we have $P^*p=p$. Thus, for every $x\in E$:
\begin{align*}
\lim_{n\to \infty}\langle P^*p_n,x \rangle=\lim_{n\to \infty}\langle p_n,P_x \rangle=\langle p,Px \rangle=\langle P^*p,x \rangle=\langle p,x \rangle.
\end{align*}
Hence, $P^*p_n\to p$ weakly$^*$ in $E^*$.
\end{proof}

Let $E$ be a (real) vector space which is endowed with an order structure defined by a reflexive, transitive, and anti-symmetric binary relation $\ge$. Then $(E,\ge)$ is called an \textit{ordered vector space} if the following conditions are satisfied: (i) $x\ge y$ implies $x+z\ge y+z$ for every $x,y,z\in E$; (ii) $x\ge y$ implies $\alpha x\ge \alpha y$ for every $x,y\in E$ and $\alpha>0$. The subset $E_+=\{ x\in E\mid x\ge 0 \}$ is called a \textit{positive cone} of an ordered vector space $E$. A subset $C$ of a vector space $E$ is called a \textit{proper cone} if $C+C\subset C$, $\alpha C\subset C$ for every $\alpha>0$, and $C\cap(-C)=\{ 0 \}$. Every positive cone of an ordered vector space is a proper cone. Conversely, every proper cone $C$ of a vector space $E$ defines, by virtue of $x\ge y \Leftrightarrow y-x\in C$, an order under which $E$ is an ordered vector with positive cone $C$. If $V$ is a vector subspace of an order vector space $E$, then the induced ordering on $V$ is determined by the proper cone $V\cap E_+$, and hence, $V_+=V\cap E_+$ is a positive cone of $V$. 

As a convention, by an \textit{ordered lcHs}, we always assume that its positive cone is closed (see \cite[V.4.(LTO)]{sw99}). Let $E$ be an ordered lcHs. A continuous linear functional $x^* \in E^*$ is said to be \textit{positive} if $\langle x,x^* \rangle \ge 0$ for every $x\in E_+$. Denote by $E^*_+$ the set of continuous linear functionals on $E$. Let $E$ and $F$ be ordered vector spaces. A linear operator $u:E\to F$ is said to be \textit{positive} if $u(E_+)\subset u(F_+)$.  

The following observation is also useful in Section \ref{sec}.

\begin{thm}
\label{pos}
Let $E$ be an ordered lcHs and $P$ be a continuous projection of $E$ onto a closed vector subspace $V$. Then $P:E\to V$ is a positive linear operator. 
\end{thm}

\begin{proof}
Choose any $x\in E_+$. If $Px\not\in V_+$, then by the separation theorem (see \cite[Theorem 3.5]{ru91}), there exists $x^*\in E^*$ such that $\langle Px,x^* \rangle=1$ and $\langle y,x^* \rangle=0$ for every $y\in V$. Let $x=x_1\oplus x_2$ be the direct sum decomposition of $x$ into $x_1\in V$ and $x_2\in \mathcal{R}(I-P)$. Since $\langle x_1,x^* \rangle=0$ and $\langle x_2,P^*x^* \rangle=0$, we have
$$
1=\langle Px,x^* \rangle=\langle x_1+x_2,P^*x^* \rangle=\langle x_1,P^*x^* \rangle=\langle Px_1,x^* \rangle=\langle x_1,x^* \rangle=0,
$$
a contradiction. Therefore, $Px\in V_+$. 
\end{proof}

\subsection{Galerkin Approximations}
\begin{quote}
A quite natural idea when considering an infinite dimensional (variational) problem is to approximate it by finite dimensional problems. This has important consequences both from theoretical (existence, etc.)\ and the numerical point of view.\ [...]\ We stress the fact that this type of finite dimensional approximation method is very flexible and can be applied [...] to a large number of linear or nonlinear problems. \cite[p.\,71]{abm14}.
\end{quote}

\begin{dfn}
A \textit{Galerkin approximation scheme} of a lcHs $E$ is a sequence $\{ V^n \}_{n\in \N}$ of finite-dimensional subspaces of $E$ such that for every $x\in E$, there exists a sequence $\{ x_n \}_{n\in \N}$ with $x_n\in V^n$ for each $n\in \N$ and $x_n\to x$ in the locally convex topology of $E$.  
\end{dfn}

The following result is fundamental.
 
\begin{prop}
\label{gal}
Let $(E,\| \cdot \|)$ be a separable Banach space. 
\begin{enumerate}[\rm (i)]
\item There exists a Galerkin approximation scheme $\{ V^n \}_{n\in \N}$ of $(E,\| \cdot \|)$ such that $V^1\subset V^2\subset\cdots$ and $\overline{\bigcup_{n\in \N}V^n}^{\,\| \cdot \|}=E$. 
\item There exists a Galerkin approximation scheme $\{ W^n \}_{n\in \N}$ of $(E^*,\mathit{w}^*)$ such that $W^1\subset W^2\subset\cdots$ and $\overline{\bigcup_{n\in \N}W^n}^{\,\mathit{w}^*}=E^*$. 
\end{enumerate}
\end{prop}

\begin{proof}
(i): For the construction of a Galerkin approximation scheme, just take a countable dense subset $\{ x_i \}_{i\in \N}$ of $E$ and let $V^n$ be the linear span of the finite set $\{ x_1,\dots,x_n\}$. See \cite[Proposition 3.1.1]{abm14} for details. 

(ii): Let $S^*$ be the closed unit ball in $E^*$. Since $E^*$ is separable with respect to the weak$^*$ topology because of the separability of $E$ (see \cite[IV.1.7]{sw99}), there is a countable set $\{ q_i \}_{i\in \N}$ in $S^*$ such that $\overline{\{ q_i \}_{i\in \N}}^{\,\mathit{w}^*}=S^*$. Let $W^n$ be the linear span of the finite set $\{ q_1,\dots,q_n \}$. By construction, $W^1\subset W^2\subset\cdots$ and $\overline{\bigcup_{n\in \N}W^n}^{\,\mathit{w}^*}=E^*$. Let $p\in E^*\setminus \{ 0 \}$ be arbitrarily fixed. We must construct $p_n\in W^n$ such that $p_n\to p$ weakly$^*$. By virtue of normalization, without loss of generality we may assume that $p\in S^*$. Since $S^*$ is metrziable with respect to the weak$^*$ topology in view of the separability of $E$ (see \cite[Theorem V.3.1]{ds58}), for a compatible metric $\delta$ on $S^*$ with the weak$^*$ topology, there exists a mapping $k\mapsto n(k)$ from $\N$ into $\N$ such that $\delta(p,q_{n(k)})<1/k$ for each $k\in \N$. Since $q_{n(k)}\in W^{n(k)}\subset W^n$ for each $k,n\in \N$ with $n\ge n(k)$, we obtain $\mathrm{dist}(p,W^n)\le \mathrm{dist}(p,W^{n(k)})<1/k$ for each $n\ge n(k)$, and hence, $\mathrm{dist}(p,W^n)\to 0$ as $n\to \infty$, which guarantees that there exists $p_n\in W^n$ such that $\delta(p, p_n)\to 0$. 
\end{proof}

\begin{thm}
\label{prj2}
Let $(E,\| \cdot \|)$ be a separable Banach space. 
\begin{enumerate}[\rm (i)]
\item If $\{ V^n \}_{n\in \N}$ is a Galerkin approximation scheme of $(E,\| \cdot \|)$ such that $V^1\subset V^2\subset\cdots$ and $\overline{\bigcup_{n\in \N}V^n}^{\,\| \cdot \|}=E$, and $P_n$ is a continuous projection of $E$ onto $V^n$, then for every $x\in E$ the sequence $\{ P_nx \}_{n\in \N}$ contains a subsequence converging weakly to $x$. 
\item If $\{ W^n \}_{n\in \N}$ is a Galerkin approximation scheme of $(E^*,\mathit{w}^*)$ such that $W^1\subset W^2\subset\cdots$ and $\overline{\bigcup_{n\in \N}W^n}^{\,\mathit{w}^*}=E^*$, and $Q_n$ is a continuous projection of $E^*$ onto $W^n$, then for every $p\in E^*$ the sequence $\{ Q_nx \}_{n\in \N}$ contains a subsequence converging weakly$^*\!$ to $p$. 
\end{enumerate}
\end{thm}

\begin{proof}
(i): Let $x\in E$ be arbitrarily fixed. It suffices to show that there is a subsequence $\{ P_mx \}_{m\in \N}$ of $\{ P_nx \}_{n\in \N}$ such that $\langle p,P_mx \rangle\to \langle p,x \rangle$ for every $p\in E^*$. By Proposition \ref{gal}(i),  there exists $x_n\in V^n$ such that $x_n\to x$. Choose any $p\in E^*$. Since $\{ (I-P_n^*)p \}_{n\in \N}$ is a bounded sequence in $E^*$ in view of $\mathcal{R}(I-P_{n+1}^*)\subset \mathcal{R}(I-P_n^*)\subset\cdots$ for each $n\in \N$ and $\{ x_i \}_{i\in \N}$ is a bounded sequence in $E$, the numerical double sequence $\{ \langle (I-P_n^*)p,x_i \rangle \}_{(i,n)\in \N^2}$ contains a convergent subsequence $\{ \langle (I-P_m^*)p,x_j \rangle \}_{(j,m)\in \N^2}$. It follows from $I-P_m^*=(V^m)^\perp$ and $x_j\in V^j$ that $\langle (I-P_m^*)p,x_j \rangle=0$ for each $j,m\in \N$ with $j\le m$. Hence, 
\begin{align*}
\lim_{m\to \infty}\sup_{j\in \N}|\langle (I-P_m^*)p,x_j \rangle|
& =\lim_{m\to \infty}\sup_{j>m}|\langle (I-P_{m}^*)p,x_{j} \rangle| \\
& =\lim_{j,m\to \infty}|\langle (I-P_{m}^*)p,x_{j} \rangle|.
\end{align*}
This means that the convergence $\lim_m\langle (I-P_m^*)p,x_j \rangle=0$ is uniform in $j$. We thus obtain
\begin{align*}
\lim_{m\to\infty}\langle p,x-P_mx \rangle=\lim_{m\to\infty}\langle p,(I-P_m)x \rangle
& =\lim_{m\to\infty}\langle (I-P_m^*)p,x \rangle \\
& =\lim_{m\to\infty}\lim_{j\to \infty}\langle (I-P_m^*)p,x_j \rangle \\
& =\lim_{j\to\infty}\lim_{m\to \infty}\langle (I-P_m^*)p,x_j \rangle=0,
\end{align*}
where the forth equality exploits the commutativity of the double limit of the double sequence (see \cite[Theorem 9.16]{ap74}). Therefore, $P_mx\to x$ weakly.  

(ii): Let $p\in E^*$ be arbitrarily fixed. It suffices to show that there is a subsequence $\{ Q_mp \}_{m\in \N}$ of $\{ Q_np \}_{n\in \N}$ such that $\langle Q_mp,x \rangle\to \langle p,x \rangle$ for every $x\in E$. By Proposition \ref{gal}(ii), there exists $p_n\in W^n$ such that $p_n\to p$ weakly$^*$. Choose any $x\in E$. Since $\{ (I-Q_n^*)x \}_{n\in \N}$ is a bounded sequence in $E^{**}$ in view of $\mathcal{R}(I-Q_{n+1}^*)\subset \mathcal{R}(I-Q_n^*)\subset\cdots$ for each $n\in \N$ and $\{ p_i \}_{i\in \N}$ is a bounded sequence in $E^*$, the numerical double sequence $\{ \langle p_i,(I-Q_n^*)x \rangle \}_{(i,n)\in \N^2}$ contains a convergent subsequence $\{ \langle p_j,(I-Q_m^*)x \rangle \}_{(j,m)\in \N^2}$. It follows from $I-Q_m^*=(W^m)^\perp$ and $p_j\in W^j$ that $\langle p_j,(I-Q_m^*)x \rangle=0$ for each $j,m\in \N$ with $j\le m$. Hence, 
\begin{align*}
\lim_{m\to \infty}\sup_{j\in \N}|\langle p_j,(I-Q_m^*)x \rangle|
& =\lim_{m\to \infty}\sup_{j\ge m}|\langle p_j,(I-Q_{m}^*)x \rangle| \\
& =\lim_{j,m\to \infty}|\langle p_j,(I-Q_{m}^*)x \rangle|.
\end{align*}
This means that the convergence $\lim_m\langle p_j,(I-Q_m^*)x \rangle=0$ is uniform in $j$. We thus obtain
\begin{align*}
\lim_{m\to\infty}\langle p-Q_mp,x \rangle=\lim_{m\to\infty}\langle (I-Q_m)p,x \rangle
& =\lim_{m\to\infty}\langle p,(I-Q_m^*)x \rangle \\
& =\lim_{m\to\infty}\lim_{j\to \infty}\langle p_j,(I-Q_m^*)x \rangle \\
& =\lim_{j\to\infty}\lim_{m\to \infty}\langle p_j,(I-Q_m^*)x \rangle=0,
\end{align*}
where the forth equality exploits the commutativity of the double limit of the double sequence. Therefore, $Q_mp\to p$ weakly$^*$.  
\end{proof}

\section{Existence of Walrasian Equilibria in Banach Spaces} 
\label{sec}
In infinite-dimensional commodity spaces, another idea for the application of Fatou's lemma is required to establish the existence of Walrasian equilibria. It is \citet{be72} who first brought the idea of Galerkin approximations into play in general equilibrium theory to deal with the $L^\infty$ commodity space in economies with finite agents.\fn{Bewley ascribes the idea to a suggestion of Jean-Fran\c{c}ois Mertens; see \cite[Acknowledgments]{be72}.}  Roughly speaking, his approach is as follows. Take a net of finite dimensional vector subspaces of $L^\infty$ directed by set inclusion and consider a net of truncated subeconomies in which a finite-dimensional vector subspace is a commodity space. Each truncated subeconomy has equilibria by the classical finite-dimensional result of Arrow--Debreu. Take the limit of the net of equilibria. Then it corresponds to a Walrasian equilibria in the original $L^\infty$ economy. 

We exemplify in this section that the idea of the finite-dimensional truncation method explored in \cite{be72} and then followed by \cite{be91,kh84b,su13} provides a natural approach to the existence of Walrasian equilibria also in large economies with a separable Banach space and $L^\infty$ without convexity assumptions on preferences by the effective combination of the Galerkin approximation scheme with finite-dimensional projections and Fatou's lemma in infinite dimensions. Thus, this offers an alternative technique to the existence result such as \cite{ks16b,ky91,le13,po97}, which employs the Gale--Nikaido lemma in infinite dimensions (see \cite{ya85}), a variant of fixed point theorems.

\subsection{Large Economies with Infinite-Dimensional Commodity Spaces with Bochner Integrals}
As before, $(T,\Sigma,\mu)$ is a (complete) finite measure space of agents. The commodity space is given by an ordered Banach space $E$. The consumption set $X(t)$ of each agent $t$ is a subset of $E_+$, which induces the consumption set mapping $t\mapsto X(t)\subset E$. The preference relation ${\succsim}(t)$ is a complete, transitive binary relation on a consumption set $X(t)$, which induces the preference mapping $t\mapsto {\succsim}(t)\subset E\times E$. We denote by $x\,{\succsim}(t)\,y$ the relation $(x,y)\in {\succsim}(t)$. Indifference ${\sim}(t)$ and strict preference ${\succ}(t)$ are defined in the same manner as in Subsection \ref{LE1}. Each agent possesses an initial endowment $\omega(t)\in X(t)$, which is the value of a Bochner integrable function $\omega:T\to E$. The economy $\E$ consists of the primitives $\E=\{ (T,\Sigma,\mu),X,\succsim,\omega \}$. 

The price space is $E^*$. Given a price $p\in E^*\setminus \{ 0 \}$, the budget set of each agent is $B(t,p)=\{ x\in X(t)\mid \langle p,x \rangle \le \langle p,\omega(t) \rangle \}$. A function $f\in L^1(\mu,E)$ is called an \textit{assignment} if $f(t)\in E$ a.e.\ $t\in T$. A function $f\in L^1(\mu,E)$ is called an \textit{assignment} if $f(t)\in X(t)$ a.e.\ $t\in T$. An assignment $f$ is called an \textit{allocation} with free disposal for $\E$ if $\int fd\mu\le \int \omega d\mu$. 
 
\begin{dfn}
A price-allocation pair $(p,f)$ is a \textit{Walrasian equilibrium} with free disposal for $\E$ if for a.e.\ $t\in T$: $f(t)\in B(t,p)$ and $x\not\in B(t,p)$ whenever $x\,{\succ}(t)\,f(t)$.
\end{dfn}

The standing hypothesis for the economy $\E$ is as follows.

\begin{assmp}
\label{assmp2}
\begin{enumerate}[(i)]
\item $X:T\twoheadrightarrow E_+$ is an integrably bounded multifunction with weakly compact, convex values.
\item $\mathrm{gph}\,{X}$ belongs to $\Sigma\otimes \mathrm{Borel}(E,\mathit{w})$. 
\item For every $t\in T$ there exists $z(t)\in X(t)$ such that $\omega(t)-z(t)\in \mathrm{int}\,E_+$. 
\item ${\succsim}(t)$ is a weakly closed subset of $X(t)\times X(t)$ for every $t\in T$. 
\item $\{ (t,x,y)\in T\times E\times E\mid x\,{\succsim}(t)\,y \}$ belongs to $\Sigma\otimes \mathrm{Borel}(E,\mathit{w})\otimes \mathrm{Borel}(E,\mathit{w})$. 
\item If $x\in X(t)$ is a satiation point for ${\succsim}(t)$, then $x\ge \omega(t)$; if $x\in X(t)$ is not a satiation point for ${\succsim}(t)$, then $x$ belongs to the weak closure of the upper contour set $\{ y\in X(t)\mid y\,{\succ}(t)\,x \}$. 
\end{enumerate}
\end{assmp}

Conditions (i) to (v) are imposed in \cite{ky91}. Unlike to Assumption \ref{assmp1}, we do not assume here the monotonicity of preferences. Instead, additional condition (vi) substitutes the monotonicity assumption and contains condition ($\star$) in the Banach space setting, which is imposed also in \cite{hi68,le13,po97}. The measurability condition (v) of preferences implies Assumption \ref{assmp1}(ii) under the separability of $E$: 
\begin{description}
\item[\rm{(v$'$)}] For every assignment $f$ and $g$: the set $\{ t\in T\mid f(t)\,{\succsim}(t)\,g(t) \}$ belongs to $\Sigma$. 
\end{description}
Let $\mathrm{proj}_T$ be the projection of $T\times E\times E$ onto $T$. Since the set $\{ t\in T\mid f(t)\,{\succsim}(t)\,g(t) \}$ coincides with the set 
$$
\mathrm{proj}_T\left( \left\{ (t,x,y)\in T\times E\times E\mid x\,{\succsim}(t)\,y \right\}\cap \mathrm{gph}(f,g) \right)
$$
it belongs to $\Sigma$ by the projection theorem (see \cite[Theorem III.23]{cv77}). 

Another important assumption on the commodity space $E$ we make below is that the norm interior of $E_+$ is nonempty. As well-known, this is an inevitable assumption to deal with infinite dimensionality in general equilibrium theory (see \cite{mz91}). 

The first main result of the essay is as follows.   

\begin{thm}
\label{WE1}
Let $(T,\Sigma,\mu)$ be a saturated finite measure space and $E$ be an ordered separable Banach space such that $\mathrm{int}\,E_+$ is nonempty. Then for every economy $\E$ satisfying Assumption \ref{assmp2}, there exists a Walrasian equilibrium $(p,f)$ with free disposal for $\E$ with $p\in E^*_+\setminus \{ 0 \}$. 
\end{thm}

\noindent The proof of this result relies on the following auxiliary result for finite-dimensional Euclidean spaces, and it may be worth underscoring why we work with it rather than others available in the literature, obvious candidates being results in \cite{au66, hi70a, hi74, sc69}. The reason lies in the monotonicity assumption on preferences. This is the only result in the literature known to us that explicitly avoids the use of this assumption by exploiting the compactness of the consumption set correspondence.\footnote{The original result in \cite{ky91} assumed convexity of preferences in separable Banach spaces. In the finite-dimensional setting, the convexity hypothesis can be removed from \cite{ky91} because of the classical Lyapunov convexity theorem. This is indeed possible since the aggregate demand multifunction is compact and convex, and the upper semicontinuity of the aggregate demand multifunction is preserved under integration from that of the individual demand multifunction. Thus, the standard application of the Gale--Nikaido lemma works perfectly in finite dimensions.} 

\begin{aux}[\citet{ky91}] Let $(T,\Sigma,\mu)$ be a nonatomic finite measure space. Suppose that the economy $\E$ with a finite-dimensional commodity space $\R^k$ satisfies the following conditions. 
\begin{enumerate}[\rm (i)]
\item $X:T\twoheadrightarrow \R^k_+$ is an integrably bounded multifunction with compact, convex values.
\item $\mathrm{gph}\,{X}$ belongs to $\Sigma\otimes \mathrm{Borel}(\R^k)$. 
\item For every $t\in T$ there exists $z(t)\in \R^k_+$ such that $\omega(t)-z(t)\in \R^k_{++}$. 
\item ${\succsim}(t)$ is a closed subset of $\R^k_+\times \R^k_+$ for every $t\in T$. 
\item $\{ (t,x,y)\in T\times \R^k\times \R^k\mid x\,{\succsim}(t)\,y \}$ belongs to $\Sigma\otimes \mathrm{Borel}(\R^k)\otimes \mathrm{Borel}(\R^k)$. 
\end{enumerate}
Then there exists a Walrasian equilibrium $(p,f)$ with free disposal for $\E$ with $p\in \R^k_+\setminus \{ 0 \}$. 
\end{aux}

\begin{proof}[Proof of Theorem \ref{WE1}]
\underline{Step 1}: By virtue of Proposition \ref{gal}(i), let $\{ V^n \}_{n\in \N}$ be a Galerkin approximation scheme of $(E,\| \cdot \|)$ such that $V^1\subset V^2\subset\cdots$ and $\overline{\bigcup_{n\in \N}V^n}^{\,\| \cdot \|}=E$, and $P_n$ be a continuous projection of $E$ onto $V^n$. Then $V^n_+=V^n\cap E_+$ is a positive cone of $V^n$. Furthermore, $P_n:E\to V^n$ is a positive linear operator by Theorem \ref{pos}. Construct a sequence of economies with a finite-dimensional truncation as follows:
\begin{itemize}
\item $X^n(t):=P_n(X(t))\subset P_n(E_+)\subset V^n_+$ is a consumption set of each agent restricted to the finite-dimensional commodity space $V^n$. 
\item ${\succsim}_n(t)$ is the restriction of the preference ${\succsim}(t)$ to $X^n(t)$, i.e., ${\succsim}_n(t):={\succsim}(t)\cap (X^n(t)\times X^n(t))$. 
\item $\omega_n(t)=P_n\omega(t)\in X_n(t)$ is the initial endowment with $\omega_n\in L^1(\mu,V^n)$. 
\item $\E_n=\{ (T,\Sigma,\mu),X_n,\succsim_n,\omega_n \}$ is a finite-dimensional truncation of economy $\E$ with commodity space $V^n$ conformed with the Galerkin approximation scheme. 
\end{itemize}

Corresponding to Assumption \ref{assmp2}, the finite-dimensional truncated economy $\E_n$ of $\E$ satisfies the following conditions: 
\begin{description}
\item[($\mathrm{i}_n$)] $X^n:T\twoheadrightarrow V^n_+$ is an integrably bounded multifunction with compact, convex values.
\item[($\mathrm{ii}_n$)]  $\mathrm{gph}\,X^n$ belongs to $\Sigma\otimes \mathrm{Borel}(V^n)$. 
\item[($\mathrm{iii}_n$)] For every $t\in T$ there exists $z_n(t)\in X^n(t)$ such that $\omega_n(t)-z_n(t)\in \mathrm{int}\,V^n_+$. 
\item[($\mathrm{iv}_n$)] ${\succsim}_n(t)$ is a closed subset of $V^n_+\times V^n_+$ for every $t\in T$. 
\item[($\mathrm{v}_n$)] $\{ (t,x,y)\in T\times V^n\times V^n\mid x\,{\succsim}(t)\,y \}\in \Sigma\otimes \mathrm{Borel}(V^n)\otimes \mathrm{Borel}(V^n)$. 
\end{description}
To verify conditions ($\mathrm{i}_n$), ($\mathrm{ii}_n$), ($\mathrm{iv}_n$)and ($\mathrm{v}_n$) is easy. To demonstrate ($\mathrm{iii}_n$), let $z:T\to E$ be a function satisfying Assumption \ref{assmp2}(ii) and set $z_n(t)=P_nz(t)$. Since $P_n$ is a continuous linear operator from $E$ onto $V^n$, by the open mapping theorem (see \cite[Theorem II.1.18]{ds58} or \cite[Corollaries 2.12]{ru91}), it is an open mapping. Thus, $P_n$ maps interior points of $E$ to those of $V^n$. Therefore, $\omega_n(t)-z_n(t)=P_n(\omega(t)-z(t))$ belongs to $\mathrm{int}\,V^n_+$. By Auxiliary Theorem, we may assume without loss of generality that for every $n\in \N$ there exists a Walrasian equilibrium $(q_n,f_n)\in ((V^n)^*_+\setminus \{ 0 \})\times L^1(\mu,V^n)$ with free disposal for $\E_n$. 

\underline{Step 2}: Invoking Theorem \ref{ks} yields that there exist $f,g\in L^1(\mu,E)$ such that
$$
f(t),g(t)\in X(t), \quad (f(t),g(t))\in \mathit{w}\text{-}\mathrm{Ls}\left\{ (f_n(t),\omega_n(t)) \right\} \quad\text{a.e.\ $t\in T$}
$$
and 
$$
\left( \int fd\mu,\int gd\mu \right)\in \mathit{w}\text{-}\mathrm{Ls}\left\{ \left( \int f_nd\mu,\int\omega_n d\mu \right) \right\}. 
$$
Hence, we can extract a subsequence from $\{ (f_n,\omega_n) \}_{n\in \N}$ in $L^1(\mu,E)\times L^1(\mu,E)$ (which we do not relabel) such that $(\int f_nd\mu,\int \omega_nd\mu)\to (\int fd\mu,\int gd\mu)$ weakly. Note also that for a.e.\ $t\in T$ we can extract a subsequence from $\{ \omega_n(t) \}_{n\in \N}$ in $E$ that converges weakly to $g(t)$. Furthermore, by virtue of Theorem \ref{prj2}(i), for a.e.\ $t\in T$ we can extract a subsequence from $\{ \omega_n(t) \}_{n\in \N}$ that converges weakly to $\omega(t)$. Therefore, $g(t)=\omega(t)$ a.e.\ $t\in T$, and hence, $\int\omega_nd\mu \to \int\omega d\mu$ weakly. Since $\int f_nd\mu\le\int \omega_nd\mu$ for each $n\in \N$, we have $\int fd\mu\le\int\omega d\mu$ at the limit. Therefore, $f$ is an allocation with free disposal for $\E$. Since $0\ne q_n\in (V^n)^*=(\mathcal{R}(I-P_n))^\perp\subset E^*$, by the Krein--Rutman theorem (see \cite[Corollary V.5.4.2]{sw99}), $q_n$ can be extended as a continuous positive linear functional to $E$. Thus, we can normalize equilibrium price for $\E_n$ such that $p_n=q_n/\| q_n \|\in \Delta^*:=\{ p\in E^*_+\mid \| p \|=1 \}$. Since $\Delta^*$ is weakly$^*$ compact, we can extract a subsequence from $\{ p_n \}_{n\in \N}$ (which we do not relabel) that converges weakly$^*$ to $p\in \Delta^*$. Therefore, $p\in E^*_+\setminus \{ 0 \}$. 

\underline{Step 3}: 
We claim that: 
\begin{quote}
For a.e.\ $t\in T$: $x\,{\succ}(t)\,f(t)$ implies that $\langle p,x \rangle>\langle p,\omega(t) \rangle$.
\end{quote}
Suppose, by way of contradiction, that there exists $A\in \Sigma$ of positive measure with the following property: for every $t\in A$ there exists $y\in X(t)$ such that $y\,{\succ}(t)\,f(t)$ and $\langle p,y \rangle\le \langle p,\omega(t) \rangle$. Since $\langle p,\omega(t) \rangle>0$ by Assumption \ref{assmp2}(iii), it follows from the weak continuity of ${\succsim}(t)$ that $\varepsilon y\,{\succ}(t)\,f(t)$ and $\langle p,\varepsilon y \rangle<\langle p,\omega(t) \rangle$ for some $\varepsilon\in (0,1)$. Hence, we may assume without loss of generality that  for every $t\in A$ there exists $y\in X(t)$ such that $y\,{\succ}(t)\,f(t)$ and $\langle p,y \rangle<\langle p,\omega(t) \rangle$. 
Define the multifunction $\Gamma:A\twoheadrightarrow E$ by 
$$
\Gamma(t)=\left\{ x\in X(t)\mid x\,{\succ}(t)\,f(t),\ \langle p,x \rangle<\langle p,\omega(t) \rangle \right\}.
$$
Then $\Gamma$ is an integrably bounded multifunction with $y\in \Gamma(t)$. We claim that $\Gamma$ has measurable graph. Since $\Gamma(t)$ is the intersection of the multifunctions defined by $\Gamma_1(t):=\{ x\in X(t)\mid x\,{\succ}(t)\,f(t) \}$ and $\Gamma_2(t):=\{ x\in E\mid \langle p,x \rangle<\langle p,\omega(t) \rangle \}$, we need to show that $\Gamma_1$ and $\Gamma_2$ have measurable graph. The graph measurability of $\Gamma_2$ follows easily from the joint measurability of the Carath\'eodory function given by $(t,x)\mapsto \langle p,x \rangle-\langle p,\omega(t) \rangle$ (see \cite[Lemma 4.51]{ab06} or \cite[Lemma 8.2.6]{af90}). To show the graph measurability of $\Gamma_1$, let $\theta:T\times E\to T\times E\times E$ be a mapping given by $\theta(t,x)=(t,x,f(t))$ and $\mathrm{proj}_{T\times E\times E}$ be a projection of $(T\times E)\times(T\times E\times E)$ onto the range space $T\times E\times E$ of $\theta$. Since $\mathrm{proj}_{T\times E\times E}(\mathrm{gph}\,\theta)$ belongs to $\Sigma\otimes \mathrm{Borel}(E,\mathit{w})\otimes \mathrm{Borel}(E,\mathit{w})$ by the projection theorem (see \cite[Theorem III.23]{cv77}), the set defined by 
$$
G:=\{ (t,x,y)\in T\times E\times E\mid x\,{\succ}(t)\,y \}\cap ((\mathrm{gph}\,X)\times E)\cap \mathrm{proj}_{T\times E\times E}(\mathrm{gph}\,\theta)
$$ 
belongs to $\Sigma\otimes \mathrm{Borel}(E,\mathit{w})\otimes \mathrm{Borel}(E,\mathit{w})$ in view of Assumption \ref{assmp2}. Since $\mathrm{gph}\,\Gamma_1=\mathrm{proj}_{T\times E}(G)$, again by the projection theorem, $\mathrm{gph}\,\Gamma_1$ belongs to $\Sigma\otimes \mathrm{Borel}(E,\mathit{w})$, where $\mathrm{proj}_{T\times E}$ is a projection of $(T\times E)\times E$ onto $T\times E$. 

Let $h:A\to E$ be a measurable selector from $\Gamma$. Define $h_n:A\to E$ by $h_n(t)=P_nh(t)$. Invoking Theorem \ref{ks}, there exists a Bochner integrable function $\hat{h}:A\to E$ such that $\hat{h}(t)\in \mathit{w}\text{-}\mathrm{Ls}\,\{ h_n(t) \}$ and $\hat{h}(t)\in X(t)$ a.e.\ $t\in A$. Hence, for a.e.\ $t\in A$ there is a subsequence of $\{ h_n(t) \}_{n\in \N}$ in $E$ converging weakly to $\hat{h}(t)$. Furthermore, by virtue of Theorem \ref{prj1}(i), there is a subsequence of $\{ h_n(t) \}_{n\in \N}$ converging weakly to $h(t)$. We thus have $\hat{h}(t)=h(t)$ a.e.\ $t\in A$, and hence, $(f(t),h(t))\in \mathit{w}\text{-}\mathrm{Ls}\,\{ (f_n(t),h_n(t)) \}$ a.e.\ $t\in A$ in view of Step 1. Suppose that the set defined by
$$
\bigcup_{n\in \N}\left\{ t\in A\mid h_n(t)\,{\succ}_n(t)\,f_n(t), \,\langle p_n,h(t)\rangle<\langle p_n,\omega(t) \rangle \right\}
$$
is of measure zero. Then for each $n\in \N$: $f_n(t)\,{\succsim}(t)\,h_n(t)$ or $\langle p_n,h(t)\rangle\ge \langle p_n,\omega(t) \rangle$ a.e.\ $t\in A$. Since $p_n\to p$ weakly$^*$, passing to the limit along a suitable subsequence of $\{ (f_n(t),h_n(t)) \}$ in $E\times E$ yields $f(t)\,{\succsim}(t)\,h(t)$ or $\langle p,h(t)\rangle\ge \langle p,\omega(t) \rangle$ a.e.\ $t\in A$, a contradiction to the fact that $h$ is a measurable selector from $\Gamma$. Therefore, there exists $n\in \N$ such that $\{ t\in A\mid h_n(t)\,{\succ}_n(t)\,f_n(t), \,\langle p_n,h(t)\rangle<\langle p_n,\omega(t) \rangle \}$ is of positive measure. This is, however, impossible because $(p_n,f_n)$ is a Walrasian equilibrium for $\E_n$. Therefore, the claim is true. 

\underline{Step 4}: It remains to show that $\langle p,f(t) \rangle\le \langle p,\omega(t) \rangle$ a.e.\ $t\in T$. If $f(t)$ is a satiation point for ${\succsim}(t)$, then Assumption \ref{assmp2}(vi) implies that $\langle p,f(t) \rangle\ge \langle p,\omega(t) \rangle$. If $f(t)$ is not a satiation point, then it belongs to the weak closure of the upper contour set $\{  y\in X(t)\mid y\,{\succ}(t)\,f(t) \}$ by Assumption \ref{assmp2}(vi). The claim shown in Step 3 implies that $\langle p,f(t) \rangle\ge \langle p,\omega(t) \rangle$. Integrating the both sides of this inequality yields $\int \langle p,f(t) \rangle d\mu\ge \int \langle p,\omega(t) \rangle d\mu$. On the other hand, as demonstrated in Step 2, $\int fd\mu\le \int \omega d\mu$, and hence, $\int  \langle p,f(t) \rangle d\mu=\int \langle p,\omega(t) \rangle d\mu$. Therefore, we must have the equality $\langle p,f(t) \rangle=\langle p,\omega(t) \rangle$ a.e.\ $t\in T$. Therefore, $(p,f)\in (E^*_+\setminus \{ 0\})\times L^1(\mu,E)$ is a Walrasian equilibrium with free disposal for $\E$. 
\end{proof}

\subsection{Large Economies with Infinite-Dimensional Commodity Spaces with Gelfand Integrals}
Let $(\Omega,\F,\nu)$ be a countably generated, $\sigma$-finite measure space. Then $L^1(\nu)$ is separable with respect to the $L^1$-norm topology (see \cite[Lemma 13.14]{ab06} or \cite[Theorem 19.2]{bi95}). The norm dual of $L^1(\nu)$ is $L^\infty(\nu)$ (see \cite[Theorem IV.8.5]{ds58}). The norm dual of $L^\infty(\nu)$ is $\mathit{ba}(\nu)$, the space of finitely additive signed measures on $\F$ of bounded variation that vanishes on $\nu$-null sets with the duality given by $\langle \pi,\varphi \rangle=\int\varphi d\pi$ for $\pi\in \mathit{ba}(\nu)$ and $\varphi\in L^\infty(\nu)$ (see \cite[Theorem IV.8.14]{ds58}). Note that $L^\infty(\nu)$ is Suslin, and hence, separable with respect to the weak$^*$ topology in view of the separability of $L^1(\nu)$. 

The commodity space is $L^\infty(\nu)$ with the order given by $f\ge g \Leftrightarrow f(t)\ge g(t)$ a.e.\ $t\in T$. The consumption set $X(t)$ of each agent $t$ is a subset of the positive cone $L^\infty_+(\nu)$, which induces the consumption set mapping $t\mapsto X(t)\subset L^\infty(\nu)$. The preference relation ${\succsim}(t)$ is a complete, transitive binary relation on a consumption set $X(t)$, which induces the preference mapping $t\mapsto {\succsim}(t)\subset L^\infty(\nu)\times L^\infty(\nu)$. Each agent possesses an initial endowment $\omega(t)\in X(t)$, which is the value of a Gelfand integrable function $\omega:T\to L^\infty(\nu)$. The (Gelfand) economy $\E$ consists of the primitives $\E=\{ (T,\Sigma,\mu),X,\succsim,\omega \}$. 

The price space is $\mathit{ba}(\nu)$. Given a price $\pi\in \mathit{ba}(\nu) \{ 0 \}$, the budget set of each agent is $B(t,\pi)=\{ x\in X(t)\mid \langle \pi,x \rangle \le \langle \pi,\omega(t) \rangle \}$. A function $f\in G^1(\mu,E)$ is called an \textit{assignment} if $f(t)\in L^\infty(\nu)$ a.e.\ $t\in T$. A function $f\in G^1(\mu,E)$ is called an \textit{assignment} if $f(t)\in X(t)$ a.e.\ $t\in T$. An assignment $f$ is called an \textit{allocation} with free disposal for $\E$ if $\mathit{w}^*\text{-}\int fd\mu\le \mathit{w}^*\text{-}\int \omega d\mu$. 

\begin{dfn}
A price-allocation pair $(\pi,f)\in (\mathit{ba}(\nu)\setminus \{ 0 \})\times G^1(\mu,L^\infty(\nu))$ is a \textit{Walrasian equilibrium} with free disposal for $\E^G$ if for a.e.\ $t\in T$: $f(t)\in B(t,\pi)$ and $x\not\in B(t,\pi)$ whenever $x\,{\succ}(t)\,f(t)$.
\end{dfn}

The norm dual $\mathit{ba}(\nu)$ of $L^\infty(\nu)$ is much larger than $L^1(\nu)$.

\begin{quote}
One could call any element of $\mathit{ba}$ a price system, but since those elements of $\mathit{ba}$ not belonging to $L^1$ have no economic interpretation, we will be interested only in equilibria with price systems in $L^1$ \cite[p.\,519]{be72}.

The theorem ... would be of little interest if one could not find interesting conditions under which equilibrium price systems could be chosen from $L^1$ \cite[p.\,523]{be72}. 
\end{quote}
Another contribution of \cite{be72} is an effective application of the Yosida--Hewitt decomposition of finitely additive measures to derive of an equilibrium price in $L^1(\nu)$ (i.e., it is countably additive). Toward this end, we make the following assumption.  

\begin{assmp}
\label{assmp3}
\begin{enumerate}[(i)]
\item $X:T\twoheadrightarrow L^\infty_+(\nu)$ is a multifunction with weakly$^*\!$ compact, convex values such that there exists a weakly$^*$ compact subset of $L^\infty(\nu)$ such that $X(t)\subset K$ for every $t\in T$. 
\item $\mathrm{gph}\,X$ belongs to $\Sigma\otimes \mathrm{Borel}(L^\infty(\nu),\mathit{w}^*)$. 
\item For every $t\in T$ there exists $z(t)\in X(t)$ such that $\omega(t)-z(t)$ belongs to the norm interior of $L^\infty_+(\nu)$. 
\item ${\succsim}(t)$ is a weakly$^*\!$ closed subset of $X(t)\times X(t)$ for every $t\in T$. 
\item $\{ (t,x,y)\in T\times E\times E\mid x\,{\succsim}(t)\,y \}$ belongs to $\Sigma\otimes \mathrm{Borel}(L^\infty(\nu),\mathit{w}^*)\otimes \mathrm{Borel}(L^\infty(\nu),\mathit{w}^*)$. 
\item If $x\in X(t)$ is a satiation point for ${\succsim}(t)$, then $x\ge \omega(t)$; if $x\in X(t)$ is not a satiation point for ${\succsim}(t)$, then $x$ belongs to the weak$^*\!$ closure of the upper contour set $\{ y\in X(t)\mid y\,{\succ}(t)\,x \}$. 
\end{enumerate}
\end{assmp}

Assumption \ref{assmp3} and the saturation of the measure space overcome the three difficulties raised in \cite[p.\,224]{be91}: 
\begin{quote}
\begin{enumerate}[(1)]
\item There exists no infinite dimensional version of Fatou's lemma, similar to Schmeidler's version in the finite dimensional case.
\item Budget sets in $L^\infty$ are typically not norm bounded and hence not weak-star compact, even when they are defined in $L^1$. [... W]e see that in this case demand functions with respect to price systems in $L^1$ need not be defined. 
\item The evaluation function $\langle \pi,x \rangle$ is not continuous on $L^\infty\times \mathit{ba}$ or on $L^\infty\times L^1$, when $L^\infty$ is given the $\sigma(L^\infty,L^1)$ topology, $\mathit{ba}$ the $\sigma(\mathit{ba},L^\infty)$ topology, and $L^1$ the $\sigma(L^1,L^\infty)$ topology. This means that if some agent's demand function $\xi$ were defined on a certain set of price systems, $\xi$ would not necessarily be $\sigma(L^\infty,L^1)$ continuous there. 
\end{enumerate}
\end{quote}

The next theorem is the second main result of the essay: it offers a completely different proof from that available in  \cite{ks16b} and removes the monotonicity and the convexity of preferences from \cite{be91,su13}. Steps 1 to 4 of the proof below follow the same argument with that of Theorem \ref{WE1} with a suitable replacement of the weak topology and the Bochner integrals by the weak$^*$ topology and the Gelfand integrals. A new aspect in the proof specific to the $L^\infty$ setting is Step 5, which follows the argument of \cite{be72}. 

\begin{thm}
\label{WE2}
Let $(T,\Sigma,\mu)$ be a saturated finite measure space and $(\Omega,\F,\nu)$ be a countably generated $\sigma$-finite measure space. Then for every economy $\E^G$ satisfying Assumption \ref{assmp3}, then there exists a Walrasian equilibrium $(\pi,f)$ with free disposal for $\E^G$ with $\pi\in L^1_+(\nu)\setminus \{ 0 \}$. 
\end{thm}

\begin{proof}
\underline{Step 1}: By virtue of Proposition \ref{gal}(ii), let $\{ W^n \}_{n\in \N}$ be a Galerkin approximation scheme of $(L^\infty(\nu),\mathit{w}^*)$ such that $W^1\subset W^2\subset\cdots$ and $L^\infty(\nu)=\overline{\bigcup_{n\in \N}W^n}^{\,\mathit{w}^*}$, and $Q_n$ be a continuous projection of $L^\infty(\nu)$ onto $W^n$. Then $W^n_+=W^n\cap L^\infty_+(\nu)$ is a positive cone of $W^n$. Furthermore, $Q_n:L^\infty(\nu)\to W^n$ is a positive linear operator. As shown in the proof of Theorem \ref{WE1}, for every $n\in \N$ there exists a Walrasian equilibrium $(\lambda_n,f_n)\in ((W^n)^*_+\setminus \{ 0 \})\times G^1(\mu,W^n)$ with free disposal for $\E_n$. 

\underline{Step 2}:
Invoking Theorem \ref{kss} yields that there exist $f,g\in G^1(\mu,L^\infty(\nu))$ such that
$$
f(t),g(t)\in X(t), \quad (f(t),g(t))\in \mathit{w}^*\text{-}\mathrm{Ls}\left\{ (f_n(t),\omega_n(t)) \right\} \quad\text{a.e.\ $t\in T$}
$$
and 
$$
\left( \mathit{w}^*\text{-}\int fd\mu,\mathit{w}^*\text{-}\int gd\mu \right)\in \mathit{w}^*\text{-}\mathrm{Ls}\left\{ \left( \mathit{w}^*\text{-}\int f_nd\mu,\mathit{w}^*\text{-}\int\omega_n d\mu \right) \right\}. 
$$
Hence, we can extract a subsequence from $\{ (f_n,\omega_n) \}_{n\in \N}$ in $G^1(\mu,L^\infty(\nu))\times G^1(\mu,L^\infty(\nu))$ (which we do not relabel) such that $(\mathit{w}^*\text{-}\int f_nd\mu,\mathit{w}^*\text{-}\int \omega_nd\mu)\to (\mathit{w}^*\text{-}\int fd\mu,\mathit{w}^*\text{-}\int gd\mu)$ weakly$^*$. Note also that for a.e.\ $t\in T$ we can extract a subsequence from $\{ \omega_n(t) \}_{n\in \N}$ in $L^\infty(\nu)$ that converges weakly$^*$ to $g(t)$. Furthermore, by virtue of Theorem \ref{prj2}(ii), for a.e.\ $t\in T$ we can extract a subsequence from $\{ \omega_n(t) \}_{n\in \N}$ that converges weakly$^*$ to $\omega(t)$. Therefore, $g(t)=\omega(t)$ a.e.\ $t\in T$, and hence, $\mathit{w}^*\text{-}\int\omega_nd\mu \to \mathit{w}^*\text{-}\int\omega d\mu$ weakly$^*$. Since $\mathit{w}^*\text{-}\int f_nd\mu\le\mathit{w}^*\text{-}\int \omega_nd\mu$ for each $n\in \N$, we have $\mathit{w}^*\text{-}\int fd\mu\le\mathit{w}^*\text{-}\int\omega d\mu$ at the limit. Therefore, $f$ is an allocation with free disposal for $\E^G$. Since $0\ne \lambda_n\in (V^n)^*=(\mathcal{R}(I-P_n))^\perp\subset L^\infty(\nu)^*=\mathit{ba}(\nu)$, by the Krein--Rutman theorem (see \cite[Corollary V.5.4.2]{sw99}), $\lambda_n$ can be extended as a continuous positive linear functional to $L^\infty(\nu)$. Thus, we can normalize equilibrium price for $\E_n$ such that $\pi_n=\lambda_n/\| \lambda_n \|\in \Delta^*:=\{ \pi\in \mathit{ba}_+(\nu)\mid \| \pi \|=1 \}$. Since $\Delta^*$ is weakly$^*$ compact, we can extract a subsequence from $\{ \pi_n \}_{n\in \N}$ (which we do not relabel) that converges weakly$^*$ to $\pi\in \Delta^*$. Therefore, $\pi\in \mathit{ba}_+(\nu)\setminus \{ 0 \}$. 

\underline{Step 3}: 
We claim that: 
\begin{quote}
For a.e.\ $t\in T$: $x\,{\succ}(t)\,f(t)$ implies that $\langle \pi,x \rangle>\langle \pi,\omega(t) \rangle$.
\end{quote}
Suppose, by way of contradiction, that there exists $A\in \Sigma$ of positive measure with the following property: for every $t\in A$ there exists $y\in X(t)$ such that $y\,{\succ}(t)\,f(t)$ and $\langle \pi,y \rangle\le \langle \pi,\omega(t) \rangle$. Since $\langle \pi,\omega(t) \rangle>0$ by Assumption \ref{assmp2}(iv), it follows from the weak$^*$ continuity of ${\succsim}(t)$ that $\varepsilon y\,{\succ}(t)\,f(t)$ and $\langle \pi,\varepsilon y \rangle<\langle \pi,\omega(t) \rangle$ for some $\varepsilon\in (0,1)$. Hence, we may assume without loss of generality that  for every $t\in A$ there exists $y\in X(t)$ such that $y\,{\succ}(t)\,f(t)$ and $\langle \pi,y \rangle<\langle \pi,\omega(t) \rangle$. 
Define the multifunction $\Gamma:A\twoheadrightarrow L^\infty(\nu)$ by 
$$
\Gamma(t)=\{ x\in X(t)\mid x\,{\succ}(t)\,f(t),\ \langle \pi,x \rangle<\langle \pi,\omega(t) \rangle \}.
$$
Then $\Gamma$ is an integrably bounded multifunction with $y\in \Gamma(t)$. We claim that $\Gamma$ has measurable graph. Since $\Gamma(t)$ is the intersection of the multifunctions defined by $\Gamma_1(t):=\{ x\in X(t)\mid x\,{\succ}(t)\,f(t) \}$ and $\Gamma_2(t):=\{ x\in L^\infty(\nu)\mid \langle p,x \rangle<\langle p,\omega(t) \rangle \}$, we need to show that $\Gamma_1$ and $\Gamma_2$ have measurable graph. Since the weakly$^*$ compact subset $K$ of $L^\infty(\nu)$ is metrziable with respect to the weak$^*$ topology in view of the separability of $L^1(\nu)$ (see \cite[Theorem 3.16]{ru91} or \cite[IV.1.7]{sw99}), $K$ is a Polish space. Then the graph measurability of $\Gamma_2$ follows easily from the joint measurability of the Carath\'eodory function on $T\times K$ given by $(t,x)\mapsto \langle p,x \rangle-\langle p,\omega(t) \rangle$ (see \cite[Lemma 4.51]{ab06} or \cite[Lemma 8.2.6]{af90}). To show the graph measurability of $\Gamma_1$, let $\theta:T\times L^\infty(\nu)\to T\times L^\infty(\nu)\times L^\infty(\nu)$ be a mapping given by $\theta(t,x)=(t,x,f(t))$ and $\mathrm{proj}_{T\times L^\infty(\nu)\times L^\infty(\nu)}$ be a projection of $(T\times L^\infty(\nu))\times(T\times L^\infty(\nu)\times L^\infty(\nu))$ onto the range space $T\times L^\infty(\nu)\times L^\infty(\nu)$ of $\theta$. Since $\mathrm{proj}_{T\times L^\infty(\nu)\times L^\infty(\nu)}(\mathrm{gph}\,\theta)$ belongs to $\Sigma\otimes \mathrm{Borel}(L^\infty(\nu),\mathit{w}^*)\otimes \mathrm{Borel}(L^\infty(\nu),\mathit{w}^*)$ by the projection theorem (see \cite[Theorem III.23]{cv77}), the set defined by 
\begin{align*}
G:=\{ (t,x,y)\in T\times L^\infty(\nu)\times L^\infty(\nu)\mid x\,{\succ}(t)\,y \}
& \cap ((\mathrm{gph}\,X)\times L^\infty(\nu)) \\
& \cap \mathrm{proj}_{T\times L^\infty(\nu)\times L^\infty(\nu)}(\mathrm{gph}\,\theta)
\end{align*}
belongs to $\Sigma\otimes \mathrm{Borel}(L^\infty(\nu),\mathit{w}^*)\otimes \mathrm{Borel}(L^\infty(\nu),\mathit{w}^*)$ in view of Assumption \ref{assmp2}. Since $\mathrm{gph}\,\Gamma_1=\mathrm{proj}_{T\times L^\infty(\nu)}(G)$, again by the projection theorem, $\mathrm{gph}\,\Gamma_1$ belongs to $\Sigma\otimes \mathrm{Borel}(L^\infty(\nu),\mathit{w})$, where $\mathrm{proj}_{T\times L^\infty(\nu)}$ is a projection of $(T\times L^\infty(\nu))\times L^\infty(\nu)$ onto $T\times L^\infty(\nu)$. 

Let $h:A\to L^\infty(\nu)$ be a measurable selector from $\Gamma$. Invoking Theorem \ref{kss}, there exists a Gelfand integrable function $\hat{h}:A\to L^\infty(\nu)$ such that $\hat{h}(t)\in \mathit{w}^*\text{-}\mathrm{Ls}\,\{ h_n(t) \}$ and $\hat{h}(t)\in X(t)$ a.e.\ $t\in A$. Hence, for a.e.\ $t\in A$ there is a subsequence of $\{ h_n(t) \}_{n\in \N}$ in $L^\infty(\nu)$ converging weakly$^*$ to $\hat{h}(t)$. Furthermore, by virtue of Theorem \ref{prj1}(ii), there is a subsequence of $\{ h_n(t) \}_{n\in \N}$ converging weakly$^*$ to $h(t)$. We thus have $\hat{h}(t)=h(t)$ a.e.\ $t\in A$, and hence, $(f(t),h(t))\in \mathit{w}^*\text{-}\mathrm{Ls}\,\{ (f_n(t),h_n(t)) \}$ a.e.\ $t\in A$ in view of Step 1. Suppose that the set defined by
$$
\bigcup_{n\in \N}\{ t\in A\mid h_n(t)\,{\succ}_n(t)\,f_n(t), \,\langle p_n,h(t)\rangle<\langle p_n,\omega(t) \rangle \}
$$
is of measure zero. Then for each $n\in \N$: $f_n(t)\,{\succsim}(t)\,h_n(t)$ or $\langle \pi_n,h(t)\rangle\ge \langle \pi_n,\omega(t) \rangle$ a.e.\ $t\in A$. Since $\pi_n\to p$ weakly$^*$, passing to the limit along a suitable subsequence of $\{ (f_n(t),h_n(t)) \}$ in $L^\infty(\nu)\times L^\infty(\nu)$ yields $f(t)\,{\succsim}(t)\,h(t)$ or $\langle p,h(t)\rangle\ge \langle p,\omega(t) \rangle$ a.e.\ $t\in A$, a contradiction to the fact that $h$ is a measurable selector from $\Gamma$. Therefore, there exists $n\in \N$ such that $\{ t\in A\mid h_n(t)\,{\succ}_n(t)\,f_n(t), \,\langle \pi_n,h(t)\rangle<\langle \pi_n,\omega(t) \rangle \}$ is of positive measure. This is, however, impossible because $(\pi_n,f_n)$ is a Walrasian equilibrium for $\E_n$. Therefore, the claim is true.

\underline{Step 4}: It remains to show that $\langle \pi,f(t) \rangle\le \langle \pi,\omega(t) \rangle$ a.e.\ $t\in T$. If $f(t)$ is a satiation point for ${\succsim}(t)$, then Assumption \ref{assmp3}(vi) implies that $\langle \pi,f(t) \rangle\ge \langle \pi,\omega(t) \rangle$. If $f(t)$ is not a satiation point, then it belongs to the weak closure of the upper contour set $\{  y\in X(t)\mid y\,{\succ}(t)\,f(t) \}$ by Assumption \ref{assmp3}(vi). The claim shown in Step 3 implies that $\langle \pi,f(t) \rangle\ge \langle \pi,\omega(t) \rangle$. Integrating the both sides of this inequality yields $\int \langle \pi,f(t) \rangle d\mu\ge \int \langle \pi,\omega(t) \rangle d\mu$. On the other hand, as demonstrated in Step 2, $\mathit{w}^*\text{-}\int fd\mu\le \mathit{w}^*\text{-}\int \omega d\mu$, and hence, $\int  \langle \pi,f(t) \rangle d\mu=\int \langle \pi,\omega(t) \rangle d\mu$. Therefore, we must have the equality $\langle \pi,f(t) \rangle=\langle \pi,\omega(t) \rangle$ a.e.\ $t\in T$. Therefore, $(\pi,f)\in (\mathit{ba}_+(\nu)\setminus \{ 0 \})\times G^1(\mu,L^\infty(\nu))$ is a Walrasian equilibrium with free disposal for $\E^G$. 

\underline{Step 5}: By the Yosida--\hspace{0pt}Hewitt decomposition of finitely additive measures (see \cite[Theorems 1.22 and 1.24]{yh52}), $\pi$ is decomposed uniquely into $\pi=\pi_1+\pi_2$, where $\pi_1\ge 0$ is countably additive and $\pi_2\ge 0$ is purely finitely additive. (Here, $\pi_2$ is \textit{purely finitely additive} if every countably additive measure $\pi'$ on $\F$ satisfying $0\le \pi'\le \pi_2$ is identically zero.) Furthermore, there exists a sequence $\{ \Omega_n \}$ in $\F$ such that (a) $\Omega_n\subset \Omega_{n+1}$ for each $n=1,2,\dots$; (b) $\lim_n\pi_1(\Omega\setminus \Omega_n)=0$; (c) $\pi_2(\Omega_n)=0$ for each $n=1,2,\dots$. 

We claim that $(\pi_1,f)$ a Walrasian equilibrium with free disposal for $\E^G$. To this end, suppose that $x\,{\succ}(t)\,f(t)$. We need to demonstrate that $\langle \pi_1,x \rangle>\langle \pi_1,\omega(t) \rangle$. It follows from the definition of Walrasian equilibria that $\langle \pi,x \rangle>\langle \pi,\omega(t) \rangle$. Take any $\varphi\in L^1(\nu)$. Then we have
$$
\left| \int x\chi_{\Omega_n}\varphi d\nu-\int x\varphi d\nu \right|\le \| x \|_\infty\int|1-\chi_{\Omega_n}||\varphi|d\nu=\| x \|_\infty\int_{\Omega\setminus \Omega_n}|\varphi|d\nu\to 0. 
$$
Hence, $x\chi_{\Omega_n}\to x$ weakly$^*$ in $L^\infty(\nu)$.\footnote{More specifically, one can show that this convergence is in the Mackey topology on $L^\infty(\nu)$; see \cite[Appendix I(24)]{be72}.} Since $\langle \pi_2,x\chi_{\Omega_n} \rangle=\int x\chi_{\Omega_n}d\pi_2=\int_{\Omega_n}xd\pi_2=0$ by condition (c), we have $\langle \pi,x\chi_{\Omega_n} \rangle=\langle \pi_1,x\chi_{\Omega_n} \rangle+\langle \pi_2,x\chi_{\Omega_n} \rangle=\langle \pi_1,x\chi_{\Omega_n} \rangle$. In view of $x\ge x\chi_{\Omega_n}$, we obtain
\begin{align*}
\langle \pi_1,x \rangle=\left\langle \pi_1,x \right\rangle\ge \left\langle \pi_1,x\chi_{\Omega_n} \right\rangle=\langle \pi,x\chi_{\Omega_n} \rangle
& >\langle \pi,\omega(t)\rangle\ge \langle \pi_1,\omega(t)\rangle
\end{align*} 
as desired. This also implies that $\pi_1\ne 0$. Since $\pi$ is absolutely continuous with respect to $\nu$, the Radon--Nikodym derivative of $\pi_1$ is an equilibrium price in $L^1_+(\nu)\setminus \{ 0 \}$. 
\end{proof}

\section{Concluding Remarks}
In conclusion, we ask ourselves and the reader as to where we go next: what lacunae and open questions  remain in the subject?  One obvious opening that needs to be closed emerges when we  ask  whether the sufficiency of the saturation property, as brought out in the principal substantive results of this essay, Theorems \ref{WE1} and \ref{WE2},  can be complemented by the necessity of the property, as has been shown in the theory of large non-atomic games, following \cite{ks09}. This is a question of some urgency because it is only after answering it that one can confidently move on to the consideration of commodity spaces that are not circumscribed by the assumptions that we have worked under: separability and the non-empty  norm-interiority  of the non-negative orthant of the commodity spaces underlying the formalizations of the economies that we investigate. As regards the former, it has already been emphasized that a commodity space with a countable dense set  is a rather severe restriction on the cardinality of a commodity space, and may well require higher-order Maharam types of agent spaces if it is to be jettisoned; see for example \cite{ks13,ks14b,ks15,ks16b}. Here, the problem of measurable selectors in non-separable spaces can hardly  be bypassed and would have to be squarely and honestly faced; see \cite{ks16b} for one attempt. 

In this connection, it is worth pointing to the fact that one of the principal, though implied, subtexts of this essay is that infinite-dimensional Walrasian general equilibrium theory has by now reached a level of mathematical maturity that a distinction between the technical and the substantive, the mathematics and the economics, is parochial and constraining. It does no favor to either subject, and the perspective to the problem that  it brings, to insist on  a  separation of the  registers in this way. As such, the very basic notion of an {\it allocation,} a concept that we have formalized through either the Bochner or the  Gelfand integral, is intimately tied to the  {\it price-system}  that is invoked to value the bundles constituting it. The clean and independent split between the operation of summation, on the one hand,  and the notion of the valuation that one has chosen to work with on the other hand, can hardly be expected to hold in an investigation of  economies  with a continuum of agents and commodities. In short, what we are drawing attention to is our total silence on Pettis integration, and the substantive implications that would have to be faced once we consider it. 

It may be worthwhile to underscore the motivation behind  Pettis integration: what is so \lq natural' about Pettis integration? As the reader has already seen, there is a remarkable parallelism in the proof of existence of Walrasian equilibria in economies modelled on  separable Banach spaces and those modelled on $L^\infty,$ a parallelism that also extends to  the proof of the exact Fatou's lemma for Banach spaces and their dual spaces. This cries out for a synthesis, one that deals with Bochner and Gelfand integration simultaneously in a single framework. It is from this point of view that one is inevitably led to a space that is inclusive of both a Banach space with its norm topology, and the dual space endowed with its weak$^*\!$ topology.  This is to say that one is inevitably led to   a locally convex Hausdorff space, and to the integration of multifunctions  taking values in such a space. 
 Indeed, as demonstrated in \cite{ks15,ks16a}, the Lyapunov convexity theorem for dual spaces of a separable Banach space follows from that in lcHs under saturation. For the application of an approximate Fatou's lemma with Pettis integrals in lcHs to Walrasian equilibrium theory, see \cite{oz94,po97} for initial exploratory attempts.
 
 We close this essay with a final remark that returns to Galerkin approximations. It is fair to say that Walrasian general equilibrium theory  in its computable manifestation has largely been confined to finite representation relating both to the set of agents and commodities. In this essay, we have illustrated the importance of finite-dimensional truncations of  an infinite-dimensional commodity space, but it stands to reason that one go the full stretch and exploit  double-barrelled truncations, ones that pertain to both the set of commodities {\it and} the set of agents.  After all, the theory of Lebesgue integration begins with simple functions, which is to say in our context, that it begins with a finite partition of the continuum of   agents, all making the same choice, and that therefore it begins with a finite set of agents.\fn{This idea was first given prominence in \cite{hi70b}; also see the text \cite{hi74}.}  We leave  this consideration of the ideal as emerging out of two truncations --- out of  a double-tremble, so to speak ---  for  future investigation.

\end{document}